\pdfoutput=1
\documentclass[12pt,a4paper]{amsart}
\usepackage[a4paper,inner=2.5cm,outer=2.5cm,top=2.5cm,bottom=2.5cm]{geometry}
\usepackage{amsmath,amssymb,amsthm,enumerate,mathtools,stmaryrd
}
\usepackage{hyperref}
\hypersetup{colorlinks=true,linkcolor=blue,citecolor=teal,filecolor=magenta,urlcolor=cyan}

\usepackage[oldstyle]{libertinus}
\usepackage{libertinust1math}

\usepackage[T1]{fontenc}
\usepackage[utf8]{inputenc}

\usepackage{makecell}
\usepackage{graphicx}

\usepackage{float}

\usepackage{extarrows}

\usepackage{xcolor}

\newcommand{\C}{\mathbb{C}}

\usepackage{url}

\usepackage[backend=bibtex,style=alphabetic,sorting=nyt,isbn=false,url=true,doi=false,maxalphanames=10,minalphanames=4,mincitenames=4,maxcitenames=10,minnames=4,maxnames=10,giveninits=true,maxbibnames=99]{biblatex}
\renewbibmacro{in:}{\ifentrytype{article}{}{\printtext{\bibstring{in}\intitlepunct}}}
\AtEveryBibitem{\iffieldundef{eprint}{}{\clearfield{url}}}
\DefineBibliographyStrings{english}{page = {},pages = {}}
\DeclareFieldFormat{eprint:arxiv}{\footnotesize
	arXiv:\,\href{https://arxiv.org/abs/\thefield{eprint}}{#1}
}
\DeclareFieldFormat{url}{\footnotesize
	\iffieldundef{doi}{url:\,\url{\thefield{url}}}
	{doi:\,\href{https://doi.org/\thefield{doi}}{\nolinkurl{\thefield{doi}}}}
}

\theoremstyle{plain}
\newtheorem{theorem}{Theorem}[section]
\newtheorem{proposition}[theorem]{Proposition}
\newtheorem{corollary}[theorem]{Corollary}
\newtheorem{lemma}[theorem]{Lemma}

\theoremstyle{definition}
\newtheorem{definition}[theorem]{Definition}

\makeatletter
\@addtoreset{proofpart}{theorem}
\makeatother
\theoremstyle{remark}
\newtheorem{remark}[theorem]{Remark}

\def\oM{\overline{\mathcal{M}}}

\newcommand{\Z}{\mathbb{Z}}

\newcommand{\cS}{\mathcal{S}}

\newcommand{\res}{\mathop{\rm res}}

\newcommand\ocM{\overline{\mathcal M}}

\newcommand{\ba}{\mathbf{a}}

\newcommand{\set}[1]{\llbracket {#1} \rrbracket}
\newcommand{\III}{{\rm III}}

\newcommand{\fc}{f}
\newcommand{\Fc}{F}

\DeclareFontFamily{U}{rcjhbltx}{}
\DeclareFontShape{U}{rcjhbltx}{m}{n}{<->rcjhbltx}{}
\DeclareSymbolFont{hebrewletters}{U}{rcjhbltx}{m}{n}

\DeclareMathSymbol{\shin}{\mathord}{hebrewletters}{152}

\newcounter{pstep}

\title[New family of Hurwitz numbers]{A new family of weighted double Hurwitz numbers and a new ELSV-type formula with $\Omega$-classes}

\author[A.~Alexandrov]{A.~Alexandrov}
\address{A.~A.: Center for Geometry and Physics, Institute for Basic Science (IBS), Pohang 37673, Korea
}
\email{alex@ibs.re.kr}

\author[B.~Bychkov]{B.~Bychkov}
\address{B.~B.: Department of Mathematics, University of Haifa, Mount Carmel, 3498838, Haifa, Israel}
\email{bbychkov@hse.ru}

\author[P.~Dunin-Barkowski]{P.~Dunin-Barkowski}
\address{P.~D.-B.: Faculty of Mathematics, HSE University, Usacheva 6, 119048 Moscow, Russia; HSE--Skoltech International Laboratory of Representation Theory and Mathematical Physics, Skoltech, Bolshoy Boulevard 30 bld. 1, 121205 Moscow, Russia; and NRC “Kurchatov Institute” -- ITEP, 117218 Moscow, Russia}
\email{ptdunin@hse.ru}

\author[M.~Kazarian]{M.~Kazarian}
\address{M.~K.: Faculty of Mathematics, HSE University, Usacheva 6, 119048 Moscow, Russia; and I.M. Krichever Scientific School, Skoltech, Bolshoy Boulevard 30 bld. 1, 121205 Moscow, Russia}
\email{kazarian@mccme.ru}

\author[S.~Shadrin]{S.~Shadrin}
\address{S.~S.: Korteweg-de Vries Institute for Mathematics, University of Amsterdam, Postbus 94248, 1090GE Amsterdam, The Netherlands}
\email{S.Shadrin@uva.nl}	

\dedicatory{Dedicated to the memory of Vladimir Igorevich Arnold (1937--2010).}

\begin{document}
	
\begin{abstract} We analyze a new family of weighted double Hurwitz numbers that was introduced as a notable example in the context of the $x-y$ duality for logarithmic topological recursion. We use this family to systematically demonstrate, refine and develop techniques that play a crucial role in the interaction of hypergeometric (Orlov--Scherbin) KP tau functions and intersection theory of moduli spaces of curves. In particular, we discuss the subtleties related to the derivation of the ELSV-type formulas in this context and derive a 
new, explicit ELSV-type formula 
in terms of the so-called $\Omega$-classes. 
\end{abstract}
	
\maketitle
	
\setcounter{tocdepth}{2}
\tableofcontents

\section{Introduction}

\emph{Topological recursion (TR)}~\cite{EO-1st} is a powerful tool in enumerative geometry and combinatorics that connects these domains to the realm of integrability, cohomological field theories, Frobenius manifolds, intersection theory of moduli spaces, etc. etc. Its recent generalization, the so-called \emph{logarithmic topological recursion (LogTR)}~\cite{ABDKS4}, which itself is a natural example for even more universal {\em generalized TR} developed in~\cite{ABDKS-degenerate-irregular}, see also a survey in~\cite{HockSha}, allows for more general input data that consists of a compact Riemann surface $\Sigma$ and a pair of functions $(x,y)$ such that $dx$ and $dy$ are meromorphic.

Most problems in enumerative combinatorics resolved by topological recursion (like maps, hypermaps, constellations, various kinds of weighted Hurwitz numbers, quantum invariants of toric knots) are KP integrable. Nowadays it is quite expected since any topological recursion on the rational spectral curve is KP integrable~\cite{ABDKS-rationalspectralKP}, and in these examples the spectral curve is rational. A more subtle property of this type of problems is that the corresponding exponential generating functions are the instances of the so-called hypergeometric KP tau functions~\cite{KMMM} (also known as Orlov--Scherbin tau functions~\cite{OrlovSch-1,OrlovSch-2}). To this end, two big families of such tau functions were studied in~\cite{bychkov2021topological} and proved to satisfy topological recursion generalizing, in particular, the results of \cite{ACEH}. These, the so-called Family I and Family II, encompass  virtually all known examples of this kind.

Logarithmic topological recursion presents more possibilities, and a new family of examples, the so-called Family III (see Definition~\ref{def:fam3} below) was proposed in~\cite[Sec.~3.11]{ABDKS4} as the one satisfying LogTR. As it was mentioned in \emph{op.~cit.}, this family of examples is interesting in several different aspects: both $x$ and $y$ have logarithmic singularities; neither of the two sides of $x-y$ duality in this case is trivial; and it was also announced that there is an ELSV-type formula for a special case inside this family of examples. The goal of this paper is to put this example under the microscope and to present all results and computations announced in~\cite[Sec.~3.11]{ABDKS4}:

\begin{description}
	\item[Proof of LogTR for Family III (Theorem~\ref{thm:fam3})] The proof that we give here refines a variety of arguments proposed in~\cite{bychkov2021topological}, that is, we analyze the local behavior of the explicit closed formulas~\cite{BDKS1} for correlation differentials near the points of potential singularities. This way we reveal some new combinatorics of the type that is also highly demanded in computations related to free energies, as in~\cite{Hock-SymplecticNonInvariance}.
	\smallskip
	\item[Derivation of a new ELSV-type formula (Theorem~\ref{prop:elsvformula} / Proposition~\ref{prop:elsvformulazr})]
	This paper also contains a new instance of an ELSV-type formula for some special cases of weighted Hurwitz numbers in Family III. This formula contains the so-called $\Omega$-classes \cite{Chiodo,GLN}, also known in the literature as Chiodo classes~\cite{LPSZ}, and the effects coming from the logarithmic singularities of $x$ and $y$ in this case lead to a variety of new subtleties that have to be taken into account.
\end{description}

\begin{remark}
	Let us also note that there exists an alternative proof of LogTR in this case, perhaps more conceptual and more general, but less explicit. It goes through the interplay of symplectic and $x-y$ dualities and is available in~\cite{ABDKS-sympl}. We note also that in~\cite[Sec.~6.5]{ABDKS-sympl} the Family III of hypergeometric KP tau functions is further extended to include more examples, but we do not need this more general definition for the purposes of this paper.
\end{remark}

\subsection{Some background}

The core of this paper is quite technical and requires a lot of care for details, though the techniques that we use here are not new and can be collected from various sources. So we feel that we have to briefly recall some context to explain why this particular family of examples (Family III) is so illuminating and does deserve a detailed discussion.

\subsubsection{Context of LogTR} LogTR was developed in~\cite{ABDKS4} to systematize a number of insights collected by Hock~\cite{hock2023xy} in attempts to apply the $x-y$ swap techniques in the presence of logarithmic singularities. Since then it got quite some applications, see e.g.,~\cite{Hock-SymplecticNonInvariance,banerjee2026gwdtinvariants5dbps,banerjee2025quantumcurvestripgeometries}. It is still a rapidly evolving direction in the huge landscape of applications of topological recursion, and we find it very important to have a fully worked out example where LogTR is structurally indispensable to indicate what are the new technicalities that one might have to deal with. Indeed, it appears to be quite non-trivial to directly trace and eliminate the potentially emerging singularities, with the computations that substantially divert from the arguments in~\cite{bychkov2021topological} and require new combinatorial insights.

\subsubsection{Context of ELSV-type formulas} One of the most impressive applications of topological recursion is a uniform proof of a variety of ELSV-type formulas~\cite{ELSV} connecting enumerative combinatorics to geometry of the moduli spaces of curves. This idea is based on Eynard's formula for differentials~\cite{Eynard-intersections} (recast in the language of cohomological field theories in~\cite{DOSS}). It was first applied in this way in the literature in~\cite{DKOSS-ELSV} in order to give a new proof of the original ELSV formula and in~\cite{JPT-1} in order to give a new proof of the Johnson--Pandharipande--Tseng formulas~\cite{JPT}. Since then it was used in a variety of cases, in particular, to prove the results that were out of reach of existing techniques in algebraic geometry, see e.g.,~\cite{rELSV-2,DKPS-qr-ELSV,BDKLM,GiaKraLew}.

LogTR can be considered a special instance of blobbed TR~\cite{BS-blobbed}, see also~\cite{alexandrov2025blobbedtopologicalrecursionkp}, and in principle~\cite[Sec.~3]{BS-blobbed} provides technique to upgrade the Eynard-type formulas in the presence of blobs. However, LogTR demands an extra term excluded in~\emph{op.~cit.}, which generates the so-called string flow and has to be carefully incorporated. In Gromov--Witten theory this procedure is known under the name \emph{materialization}~\cite{Givental-semisimple}, see also~\cite{Janda-P1,alexandrov2025newspinpolynomialrelations}.

\subsubsection{Context of~\texorpdfstring{$\Omega$}{Omega}-classes}
The so-called $\Omega$-classes~\cite{Bini,Chiodo} become more and more ubiquitous in a variety of applications and are subject of many interesting properties and open conjectures~\cite{GLN,BorotKarevLewanski,blot2025cohomologicalrepresentationsquantumtau}. They are used, for instance, in ELSV-type formulas (see the references above), in explicit formulas for geometrically defined cycles on moduli spaces~\cite{JPPZ}, and in constructions of new types of integrable systems~\cite{blot2025cohomologicalrepresentationsquantumtau,BLS-inprep}, just to name a few.

As the above references indicate, multiple explicit links of the $\Omega$ classes to hypergeometric KP tau functions highlight more and more of their properties, and the new formula that we present here can definitely be used for their further study.

\subsection{Organization of the paper} In Section~\ref{sec:LogTR} we very briefly recall the definition of LogTR. In Section~\ref{sec:FamIII} we introduce the main object of study in this paper: the Family III of weighted double Hurwitz numbers and discuss their basic properties. In Section~\ref{sec:ELSV} we prove a new ELSV-type formula for a particular member of Family~III. In Appendix~\ref{sec:logProj} we prove the logarithmic projection property for Family III in a direct combinatorial way in the vein of~\cite{bychkov2021topological}. Appendix~\ref{sec:formulas} is devoted to computation of coefficients for Theorem~\ref{prop:elsvformula} given explicitly by Proposition~\ref{prop:elsvformulazr}.

\subsection{Notation} Throughout the text we use the following notation:
\begin{itemize}
	\item $\set{n}$ denotes $\{1,\dots,n\}$.
	\item $z_I$ denotes $\{z_i\}_{i\in I}$ for $I\subseteq \set{n}$.
	\item $[u^d]$ denotes the operator that extracts the corresponding coefficient from the whole expression to the right of it, that is, $[u^d]\sum_{i=-\infty}^\infty a_iu^i \coloneqq a_d$.
	\item $\cS(u)$ denotes $u^{-1}(e^{u/2} - e^{-u/2})$.
\end{itemize}

\subsection{Acknowledgments} A.~A. was supported by the Institute for Basic Science (IBS-R003-D1). B.~B. was supported by the ISF Grant 2848/25. P.~D.-B. and M.~K. were supported by the Basic Research Program of HSE University. S.~S. was supported by the Dutch Research Council. 

The paper of Vladimir Igorevich Arnold~\cite{Arnold} is a source that can justly be regarded as having inspired and provided the key ideas for the now flourishing field at the intersection of enumerative combinatorics and algebraic geometry through Hurwitz-type problems and ELSV-type formulas. We take this opportunity to dedicate the present paper to his memory.

\section{Definition of logarithmic topological recursion} \label{sec:LogTR}

In this section we recall the definition of logarithmic topological recursion (LogTR) introduced in~\cite{ABDKS4}.

Let $x$ and $y$ be functions on a compact Riemann surface $\Sigma$ with a symplectic basis of $\mathfrak{A}$ and $\mathfrak{B}$ cycles, and $\{\omega^{(g)}_{n}\}$, $g\geq 0$, $n\geq 1$, $2g-2+n>0$ be a given system of symmetric meromorphic $n$-differentials. We also define $\omega^{(0)}_1 \coloneqq ydx$ and $\omega^{(0)}_2\coloneqq B$, where $B$ is the so-called Bergman kernel, that is, the unique meromorphic bi-differential with a double pole along the diagonal with bi-residue $1$ and no further singularities, whose $\mathfrak{A}$-periods are all equal to $0$.

Assume that
\begin{itemize}
	\item $dx$ and $dy$ are meromorphic with possibly non-vanishing residues at some points (that is, $x$ and $y$ are possibly multivalued with logarithmic singularities);
	\item all zeros of $dx$ are simple;
	\item $y$ is regular at the zero locus of $dx$;
	\item the zero loci of $dx$ and $dy$ are disjoint.
\end{itemize}
Denote the zeros of $dx$ by $p_1,\dots,p_N\in \Sigma$ and let $\sigma_i$ be the deck transformations of $x$ near $p_i$.

\begin{definition} We say that the system of symmetric meromorphic differentials $\{\omega^{(g)}_n\}$, $g\geq 0$, $n\geq 1$ satisfies
	\begin{itemize}
		\item the \emph{linear loop equations}, if for any $g,n\geq 0$, $i=1,\dots,N$
		\begin{equation} \label{eq:LLE-original}
			\omega^{(g)}_{n+1}(z_{\llbracket n \rrbracket},z) + \omega^{(g)}_{n+1}(z_{\llbracket n \rrbracket},\sigma_i(z))
		\end{equation}
		is holomorphic at $z=p_i$ and has at least a simple zero in $z$ at $z=p_i$;
		\item the \emph{quadratic loop equations}, if for any $g,n\geq 0$, $i=1,\dots,N$, the quadratic differential in $z$
		\begin{equation} \label{eq:QLE-original}
			\omega^{(g-1)}_{n+2}(z_{\llbracket n \rrbracket},z,\sigma_i(z)) + \sum_{\substack{g_1+g_2=g\\ I_1\sqcup I_2 = \llbracket n \rrbracket
			}} \omega^{(g_1)}_{|I_1|+1} (z_{I_1},z)\omega^{(g_2)}_{|I_2|+1} (z_{I_2},\sigma_i(z))
		\end{equation}
		is holomorphic at $z=p_i$ and has at least a double zero at $z=p_i$.
	\end{itemize}
\end{definition}

\begin{definition}
	We say that the primitive $y$ of the differential $dy$ possesses a \emph{logarithmic singularity} at some point~$a$ on~$\Sigma$ if $dy$ has a pole at~$a$ with nonzero residue. A logarithmic singularity of~$y$ is called \emph{LogTR-vital} if this pole of $dy$ is simple and $dx$ has no pole at this point.
\end{definition}

Let $a_1,\dots,a_M$ be the LogTR-vital singular points of $y$. We denote the residues of~$dy$ at these points by $\alpha_1^{-1},\dots,\alpha_M^{-1}$, respectively. That is, the principal part of $dy$ near $a_i$ is given by $\alpha_i^{-1} dz / (z-a_i)$, where $z-a_i$ is any local coordinate. %
In the neighborhoods of $a_i$, $i=1,\dots,M$ consider the germs defined as the principal parts of coefficients of positive powers of $\hbar$ in the following expressions:
\begin{align}\label{eq:PrincipalPartsLog}
	\left( \frac{1}{\alpha_i\cS(\alpha_i\hbar \partial_{x})}\log(z-a_i)\right)dx.
\end{align}
These germs contribute to the second term in the definition below.

\begin{definition}\label{def:log-proj} We say that a collection of differentials $\{\omega^{(g)}_{n}\}$ satisfies the \emph{logarithmic topological recursion} (LogTR) on the spectral curve $(\Sigma,x,y)$, if $\omega^{(0)}_1 = ydx$, $\omega^{(0)}_2 = B$, $\{\omega^{(g)}_{n}\}$ satisfy the linear and quadratic loop equations and additionally they satisfy the so-called \emph{logarithmic projection property}: for any $g\geq 0$, $n\geq 1$, $2g-2+n>0$
	\begin{align} \label{eq:LogProjection}
		\omega^{(g)}_{n}(z_{\set{n}}) & = \sum_{i_1,\dots,i_n=1}^N \bigg(\prod_{j=1}^n \res_{z_j'= p_{i_j}} \int^{z_j'}_{p_{i_j}} B(\cdot,z_j)\bigg) \omega^{(g)}_{n}(z'_{\set{n}})
		\\ \notag & \quad
		+ \delta_{n,1}[\hbar^{2g}] \sum_{i=1}^M  \res_{z'= a_i} \bigg(\int^{z'}_{a_i} B(\cdot,z_1)\bigg)
		\left(\frac{1}{\alpha_i\cS(\alpha_i\hbar \partial_{x'})}\log(z'-a_i)\right)dx'.
	\end{align}
\end{definition}

This definition is presented in a more detailed way in~\cite{ABDKS4}.
Note that we dropped the prefix ${}^{\log{}}$ in the notation for $\omega$'s in the present paper as compared to~\cite{ABDKS4}, since LogTR is the only kind of topological recursion that we discuss here.

\section{New Family~III of weighted double Hurwitz numbers} \label{sec:FamIII}

Let $\hat y(z,\hbar)$ and $\hat\psi(\theta,\hbar)$ be two formal power series in two variables such that $\hat y(0,\hbar)=0$ and $\hat\psi(0,0)=0$. We also consider $\psi(\theta)\coloneqq \hat\psi(\theta,0)$ and $y(z) \coloneqq \hat y(z,0)$.

\begin{definition} The Orlov--Scherbin (or hypergeometric) tau function
	is given by the sum over partitions
	\begin{equation}\label{eq:ZOS}
		Z=\sum_{\lambda}\shin_\lambda(p)\shin_\lambda(\hbar^{-1} q)\exp\bigg(\sum_{(i,j)\in \lambda}\hat\psi(\hbar(i-j),\hbar)\bigg),
	\end{equation}
	where $\shin_\lambda(p)$ is the Schur function in the variables $p=(p_1,p_2,\dots)$ labeled by the partition $\lambda$ and the parameters $q_i=q_i(\hbar)$ for the second Schur function are the expansion coefficients defined by $\hat y(z,\hbar)=\sum_{j=1}^\infty q_j(\hbar)z^j$.
\end{definition}

\begin{definition}\label{def:OSdifferentials} The Orlov--Scherbin differentials are defined as
	\begin{align}\label{OSdif}
		\sum_{g=0}^\infty \hbar^{2g-2+n} \omega^{(g)}_n =\left.\prod_{i=1}^n \bigg(d_i \sum_{k_i=1}^\infty X_i^{k_i} \partial_{p_{k_i}} \bigg)\log Z\right|_{p_j=0,\, j=1,2,\dots} + \delta_{n,2}\frac{d_1X_1 d_2 X_2}{(X_1-X_2)^2},	
	\end{align}
where $X_i=X(z_i) = e^{x(z_i)} = z_i \exp(-\psi(y(z_i)))$.
\end{definition}

It is proved in~\cite{BDKS1,bychkov2021topological} that the Orlov--Scherbin differentials \eqref{OSdif} can be considered as expansions of the differentials globally defined on $\mathbb{C}P^1$ (with an affine coordinate $z$) near $z_1=\dots=z_n=0$ in the local coordinates $X_i$.

\begin{definition}[\cite{ABDKS4}] \label{def:fam3}
	The \emph{Family III} of Orlov--Scherbin input data is defined as follows:
	\begin{align} \label{eq:fam3psi}
		\psi(\theta) &\coloneqq \alpha(e^\theta-1), &\hat\psi(\theta,\hbar) &\coloneqq \cS(\hbar\partial_\theta) \psi(\theta),\\ \label{eq:fam3y}
		y(z) &\coloneqq \log R(z), &\hat y(z,\hbar) &\coloneqq \tfrac{1}{\cS(\hbar z\partial_z)} y(z),\\ \label{eq:fam3xz}
		x(z) &\coloneqq \log z -\psi(y(z))
		= \log z - \alpha\left(R(z) -1\right), & &
	\end{align}
	where $\alpha$ is a constant and $R(z)$ is a non-constant rational function such that
	\begin{itemize}
		\item $R(z)$ has simple zeros and poles;
		\item $R(0)=1$;
		\item zeros of $dx$ are simple;
		\item zeros of $dx$ do not coincide with the zeros of $R(z)$.
	\end{itemize}	
\end{definition}

\begin{definition}\label{def:HurwDef} The numbers
	\begin{equation}		
		h^{(g)}_{\mu_1,\dots,\mu_n}\coloneqq %
		\frac{\partial }{\partial p_{\mu_1}}\dots \frac{\partial }{\partial p_{\mu_n}}[\hbar^{2g-2+n}]\log Z\Bigg|_{p_j=0,\, j=1,2,\dots}
	\end{equation}
	defined via the Orlov--Scherbin tau function with the parameters $\hat \psi$ and $\hat y$ as in Definition~\ref{def:fam3} are weighted double Hurwitz numbers in the sense of~\cite{guay2017generating,harnad,als} and we call them Family III of weighted double Hurwitz numbers. 
\end{definition}

These numbers have the combinatorial meaning of certain weighted counts of paths in the Cayley graph of the symmetric group, via a proper generalization of the framework of \cite{guay2017generating, harnad}. Family I and Family II of weighted double Hurwitz numbers were introduced in~\cite{bychkov2021topological}, for the general description of all three families see \cite[Section 6]{ABDKS-sympl}.

\begin{remark} While (a straightforward generalization of) Guay-Paquet--Harnad's framework allows to express the respective numbers for any given representative of Family III as weighted counts of paths in the Cayley graph of the symmetric group as mentioned above, these weighted counts are bulky and not particularly nice looking; for this reason we do not provide them explicitly in the present paper. We do not know any natural  (nicely formulated) enumerative problem that would lead to the Family III; it is an interesting open question. So far this Family III is merely a testing ground that allows us to test the new techniques brought by the theory of LogTR as opposed to the original TR.
\end{remark}

In the present paper we give a direct proof in the vein of~\cite{bychkov2021topological} of the following Theorem.
\begin{theorem}\label{thm:fam3}
	
	With the input data $\hat \psi$ and $\hat y$ of Definition~\ref{def:fam3}, the Orlov--Scherbin differentials
	$\{\omega^{(g)}_n\}$ of Definition~\ref{def:OSdifferentials} satisfy LogTR on the spectral curve $(\mathbb{C}P^1,x(z),y(z))$.
\end{theorem}

\begin{remark}
	The assumption on the simplicity of zeros and poles of $R(z)$ and zeros of $dx$ can be lifted and the statement of the theorem remains true for the generalized TR instead of LogTR and modified $\hat{y}$, see \cite{ABDKS-sympl,ABDKS-degenerate-irregular}.
\end{remark}

\begin{remark} One of the interesting features of Family III as compared with the other two general families of weighted double Hurwitz numbers discussed in~\cite{bychkov2021topological} is the following.
	For Family~I of~\cite{bychkov2021topological} all terms in $\psi(\theta)$  except the logarithmic and linear ones have to be $\hbar$-deformed in $\hat\psi$ via the application of $\cS(\hbar\partial_\theta)$, while $y(z)$ remains $\hbar$-undeformed. For Family~II, on the other hand, $\psi(\theta)$ is very simple and $\hbar$-undeformed, while for $y(z)$ the logarithmic terms have to be $\hbar$-deformed via the application of $1/\cS(\hbar z\partial_z)$. Note that in the case of Family~III both $\psi(\theta)$ and $y(z)$ have to be $\hbar$-deformed with exactly the same operators which were used in Family~I and Family~II for $\psi(\theta)$ and $y(z)$ respectively.
\end{remark}

\begin{remark}
	Note that Family~III is basically the simplest case when the spectral curve for the Orlov--Scherbin differentials has LogTR-vital singularities. Indeed, it turns out that the respective $n$-point differentials for the natural choice of $\hbar$-deformation of $\psi$ and $y$ satisfy the LogTR and not the original TR, while we expect that no $\hbar$-deformations of $\psi$ and $y$ allow to restore the original TR.
	
\end{remark}

By Definition \ref{def:log-proj}, in order to prove Theorem~\ref{thm:fam3}, we need to establish loop equations and the logarithmic projection property. The loop equations are straightforward:

\begin{proposition}\label{prop:fam3loopeq}
	The $n$-point differentials $\{\omega^{(g)}_n\}$ of Theorem~\ref{thm:fam3} satisfy the loop equations.
\end{proposition}
\begin{proof}
	This is a direct corollary of~\cite[Theorem~2.21]{bychkov2021topological}, as it is easy to see that the \emph{natural analytic assumptions} of~\cite{bychkov2021topological} are satisfied in this case.
\end{proof}

We prove the logarithmic projection property and hence Theorem \ref{thm:fam3} in Appendix~\ref{sec:logProj}. 

\section{ELSV-type formula for family III}\label{sec:ELSV}

\subsection{Classical ELSV-type formulas}
Consider LogTR, which in this case is equivalent to the original TR, for the spectral curve
\begin{equation}\label{eq:Chiodo-rs-curve}
	x=\log z-z^r,\quad y=z^s,\quad B(z_1,z_2)=\frac{dz_1dz_2}{(z_1-z_2)^2}
\end{equation}
for some positive integers $r,s$. By \cite{LPSZ} its differentials are expressed in terms of intersection numbers involving Chiodo classes (for $2g-2+n>0$):
\begin{equation}\label{eq:LPSZ}
	\omega^{(g)}_{n}(z_{\set n})=C^{2-2g-n}\sum_{(a_1,k_1),\dots,(a_n,k_n)}\int\limits_{\ocM_{g,n}}\Omega_{g,n}^{r,s}(a_1,\dots,a_n)\psi_1^{k_1}\dots\psi_n^{k_n}
	\;\prod_{i=1}^n d\xi_{a_i,k_i}(z_i).
\end{equation}
Here $a_i\in\{1,\dots,r\}$, $k_i\in\Z_{\ge0}$, $\Omega_{g,n}^{r,s}$ is the Chiodo class, and
\begin{equation}
	\xi_{a,k}(z)\coloneqq r^{\frac{r-a}{r}}\Bigl(\frac{d}{r\,d x(z)}\Bigr)^k\frac{z^{r-a}}{1-r\,z^r},
	\quad \quad
C\coloneqq\frac{s}{r^{1+s/r}}.
\end{equation}

The function $X=e^x=z\,e^{-z^r}$ can be taken as a local coordinate at the point $z=0$ of the spectral curve. Using the known expansions of the basic $\xi$-functions in this local coordinate,
\begin{equation}\label{eq:xi-expansion}
	\xi_{r-a,k}=\sum_{m=0}^\infty\frac{\bigl(m+\tfrac{a}{r}\bigr)^{m+k}}{m!}r^{m+\frac{a}{r}}X^{m r+a}
\end{equation}
from~\eqref{eq:LPSZ} we get  the corresponding expansion of the LogTR differentials
\begin{equation}\label{eq:oomega-X-expansion}
	\omega^{(g)}_n(z_{\set n})=\sum_{\mu_1,\dots,\mu_n=1}^\infty h^{(g)}_{\mu_1,\dots,\mu_n} \prod_{i=1}^n dX_i^{\mu_i},
\end{equation}
where $X_i=X(z_i)=z_ie^{-z_i^r}$ and
\begin{equation}\label{eq:JPT}
	h^{(g)}_{\mu}=C^{2-2g-n} \,r^{\frac{1}{r}\,\sum_{i=1}^n \mu_i} \prod_{j=1}^n\frac
	{{(\mu_j/r)}^{\lfloor \mu_j/r\rfloor}}
	{\lfloor \mu_j/r\rfloor!}
	\int_{\ocM_{g,n}}\frac{\Omega_{g,n}^{r,s}(\ba)}{\prod_{i=1}^n(1-\frac{\mu_i}{r}\psi_i)}.
\end{equation}
Here $\mu=(\mu_1,\dots,\mu_n)$, $\ba=(a_1,\dots,a_n)$ with $a_i=r-r \langle\mu_i/r\rangle$; $\lfloor\cdot\rfloor$ and $\langle\cdot\rangle$ denote the integer and the fractional part of a number, respectively.

If $s=1$, then the coefficients $h^{(g)}_{\mu_1,\dots,\mu_n}$ defined by the expansion~\eqref{eq:oomega-X-expansion} are called the \emph{$r$-spin Hurwitz numbers}~\cite{Zvonkine,SSZ}. The ELSV-type formula~\eqref{eq:JPT} is known as the ELSV formula for $r=s=1$~\cite{ELSV}, the Johnson--Pandharipande--Tseng formula for $r=s$~\cite{JPT,LPSZ}, Zvonkine's $qr$-ELSV formula for any $r\geq 1$ divisible by $s$~\cite{rELSV-1,rELSV-2,DKPS-qr-ELSV}.

\begin{remark}
	It is possible to get rid of unnecessary rescaling coefficients in the above relations by a more convenient choice of normalization for the spectral curve. Namely, if we take
	\begin{align}
		\tilde x&=r\log z-z^r, & \tilde y & =z^s/s,
		& \tilde X&=e^{\tilde x/r}=z\,e^{-z^r/r},
		& \tilde \xi_{a,k}(z)&=\Bigl(\frac{d}{d \tilde x(z)}\Bigr)^k\frac{z^{r-a}}{1-z^r},
	\end{align}
	then the formulas look nicer with fewer coefficients:
	\begin{align}\label{eq:rescom}
		\tilde \omega^{(g)}_{n}&=\sum_{(a_1,k_1),\dots,(a_n,k_n)}\int_{\oM_{g,n}}\Omega_{g,n}^{r,s}(\ba)\psi_1^{k_1}\dots\psi_n^{k_n}
		\prod_{i=1}^n d\tilde \xi_{a_i,k_i}(z_i)
		=\sum_{\mu_1,\dots,\mu_n=1}^\infty \tilde h^{(g)}_{\mu_1,\dots,\mu_n}\prod_{i=1}^n d\tilde X_i^{\mu_i},	\\	
		\tilde \xi_{r-a,k}&=\sum_{m=0}^\infty\frac{\bigl(m+\tfrac{a}{r}\bigr)^{m+k}}{m!} \tilde X^{m r+a}, \\
		\tilde h^{(g)}_{\mu}&=\prod_{j=1}^n\frac
		{{(\mu_j/r)}^{\lfloor \mu_j/r\rfloor}}
		{\lfloor \mu_j/r\rfloor!}
		\int_{\ocM_{g,n}}\frac{\Omega_{g,n}^{r,s}(\ba)}{\prod_{i=1}^n(1-\frac{\mu_i}{r}\psi_i)}.
	\end{align}
	However, for consistency, we will use in what follows the convention of~\cite{LPSZ} and~\eqref{eq:Chiodo-rs-curve} as the equation of the basic spectral curve (also the spectral curve~\eqref{eq:Chiodo-rs-curve} fits into the Orlov-Scherbin Family~I of~\cite{bychkov2021topological}).
\end{remark}

\subsection{Generalized logarithmic \texorpdfstring{$r$}{r}-spin Hurwitz numbers}

We would like to generalize the ELSV-type formula~\eqref{eq:JPT} to the case of the spectral curve with the same function~$x$ as in~\eqref{eq:Chiodo-rs-curve} but with a different~$y$, now with logarithmic singularities. As an example, we consider the spectral curve
\begin{equation}\label{eq:curvezr}
	x=\log z-z^r,\quad y=\log(1+z^r).
\end{equation}

The differential $dy$ has simple poles at the logarithmic singularities of $y$ and its residues
\begin{equation}
	\alpha^{-1}=\res_{z=a}dy=1
\end{equation}
are independent of the choice of singular point $a$ such that $a^r=-1$. 

Introduce the ($\hbar$-expanded) $1$-differential
\begin{equation}\label{eq:Phi}
	\Phi(z)=(y(z)-z^r)dx(z)+\sum_{a: a^r=-1}
	\res_{\tilde z=a}
	\Bigl(\Bigl(\frac{1}{\cS(\alpha\hbar\partial_{x(\tilde z)})}-1\Bigr)y(\tilde z)\Bigr)dx(\tilde z)
	\int^{\tilde z}B(\cdot,z).
\end{equation}
Define also the convolution operation of two meromorphic $1$-differentials $\mu,\nu$ by
\begin{equation}
	\mu*\nu=-\nu*\mu=\sum_{q: rq^r=1}\res_{z=q}\mu(z)\int^z\nu(\cdot).
\end{equation}
It is well defined if the differentials $\mu,\nu$ have trivial residues at the branch points $q:\;rq^r=1$.
Then, the recomputation procedure expresses the differentials ${}^{\III}\omega^{(g)}_n$ of logarithmic topological recursion for the spectral curve~\eqref{eq:curvezr} in terms of those $\omega^{(g)}_n$ associated with the spectral curve~\eqref{eq:Chiodo-rs-curve} for $s=r$ (with the same~$x$ but different~$y$), see \cite[Proposition 4.5]{ABDKS4} for details:
\begin{equation}\label{eq:logTR-recomp}
	{}^{\III}\omega^{(g)}_n(z_{\set{n}})=[\hbar^{2g}]\Bigl(\delta_{n,1}\Phi(z_1)+
	\sum_{p=0}^\infty\frac{1}{p!}\sum_{g_0=0}^\infty\hbar^{2g_0}\omega^{(g_0)}_{n+p}(z_{\set{n}},\dots)({*}\Phi)^p\Bigr).
\end{equation}
Namely, the first summand in~\eqref{eq:Phi} corresponds to passing from $y=z^r$ to $y=\log(1+z^r)$, and the second summand describes passing from the original TR to LogTR.

Next, from $\Z_r$-symmetry, we observe that $d\xi_{a,k}*\Phi=0$ for $a\ne0\bmod r$. Denote
\begin{equation}
	\fc_k\coloneqq d\xi_{r,k}*\Phi.
\end{equation}
Note that $\fc_k$ is a power series in $\hbar^2$ with numerical coefficients satisfying $\fc_0\bigm|_{\hbar=0}=0$ and $\fc_1\bigm|_{\hbar=0}\ne C$. Then, denoting
\begin{equation}
	\Fc(\psi)=\sum_{k=0}^\infty \fc_k\psi^k
\end{equation}
we get from~\eqref{eq:LPSZ} and~\eqref{eq:logTR-recomp}
\begin{multline}\label{eq:log-omega}
	{}^{\III}\omega^{(g)}_n(z_{\set{n}})=
	\sum_{g_0=0}^g[\hbar^{2(g-g_0)}]
	\sum_{p=0}^\infty\frac{C^{2-2g_0-n-p}}{p!}\\
	\times\sum_{(a_1,k_1),\dots,(a_n,k_n)}
	\int\limits_{\ocM_{g_0,n+p}}
	\Omega^{r,r}_{g_0,n+p}(\ba,r^p)\prod_{j=1}^p \Fc(\psi_{n+j})
	\;\prod_{i=1}^n \psi_i^{k_i} d\xi_{a_i,k_i}(z_i),
	\quad n\ge3.
\end{multline}
Here $(\ba,r^p)=(a_1,\dots,a_n,\underbrace{r,\dots,r}_p)$. The restriction $n\ge3$ is due to non-existence of unstable moduli spaces $\ocM_{0,1},\ocM_{0,2}$ and also due to the first summand in~\eqref{eq:logTR-recomp} that does not belong to the span of $d\xi$-differentials. It is possible to write down the contribution of these terms explicitly, but for simplicity of exposition we restrict ourselves to the case $n\ge3$ here and below in this section.

Let us introduce the \emph{generalized logarithmic $r$-spin Hurwitz numbers}  ${}^{\III}h^{(g)}_{\mu_1,\dots,\mu_n}$ as the coefficients of the expansion of the differentials ${}^{\III}\omega^{(g)}_n$ in the local coordinates $X_i=e^{x(z_i)}$ at the origin:
\begin{equation}\label{eq:hurwIII}
	{}^{\III}\omega^{(g)}_n=\sum_{\mu_1,\dots,\mu_n=1}^\infty {}^{\III} h^{(g)}_{\mu_1,\dots,\mu_n}\prod_{i=1}^ndX_i^{\mu_i}.
\end{equation}
By Theorem~\ref{thm:fam3} they are exactly the special case of Family III of weighted double Hurwitz numbers of Definition~\ref{def:HurwDef} for $R(z)=1+z^r$ and $\alpha=1$
. As mentioned in Section \ref{sec:FamIII}, these numbers have combinatorial meaning of certain weighted counts of paths in the Cayley graph of the symmetric group, via a proper generalization of Guay-Paquet--Harnad's framework~\cite{guay2017generating, harnad}.

Then, expanding both sides of~\eqref{eq:log-omega} and using~\eqref{eq:xi-expansion} we obtain (for $n\geq 3$)
\begin{equation}\label{eq:logJPT-preliminary}
	{}^{\III}h^{(g)}_{\mu}=
	r^{\frac{1}{r}\sum_{i=1}^n \mu_i} \prod_{j=1}^n\frac{{(\mu_j/r)}^{\lfloor \mu_j/r\rfloor}}{\lfloor \mu_j/r\rfloor!}
	\sum_{g_0=0}^g[\hbar^{2(g-g_0)}]
	\sum_{p=0}^\infty\frac{C^{2-2g_0-n-p}}{p!}
	\int\limits_{\ocM_{g_0,n+p}}\frac{\Omega_{g_0,n+p}^{r,r}(\ba,r^p)\prod_{j=1}^p \Fc(\psi_{n+j})}{\prod_{i=1}^n(1-\frac{\mu_i}{r}\psi_i)}.
\end{equation}
We consider the obtained equality as a preliminary version of the required generalized ELSV formula. The final one is obtained by pushing forward the class under integration to $\ocM_{g_0,n}$.

\begin{theorem}\label{prop:elsvformula}
	For $n\ge3$ the generalized logarithmic $r$-spin Hurwitz numbers are given by 
	\begin{equation}\label{eq:mainformula}
		{}^{\III}h^{(g)}_{\mu}=
		r^{\frac{1}{r}\sum_{i=1}^n \mu_i} \prod_{j=1}^n\frac{{(\mu_j/r)}^{\lfloor \mu_j/r\rfloor}}{\lfloor \mu_j/r\rfloor!}
		\sum_{g_0=0}^g[\hbar^{2(g-g_0)}]
		{C^*}^{2-2g_0-n}e^{\bigl(\sum\limits_{i=1}^n \frac{\mu_i}{r}\bigr)u^*}
		\int\limits_{\ocM_{g_0,n}}\frac{e^{\sum_{k=1}^\infty s^*_k\kappa_k}\Omega_{g_0,n}^{r,r}(\ba)}
		{\prod_{i=1}^n(1-\tfrac{\mu_i}{r}\psi_i)},
	\end{equation}
	where $\ba=(a_1,\dots,a_n)$ with $a_i=r-r \langle\mu_i/r\rangle$ and where
	$C^*$, $u^*$ and $s^*_k$ are certain power series in $\hbar^2$. Namely, $u^*$ is a solution of implicit functional equation
	\begin{equation}\label{eq:ustarorigdef}
		C\,u^*=\sum_{m=0}^\infty \fc_m\frac{{u^*}^m}{m!},
	\end{equation}
	and the other $\hbar^2$-series are expressed in terms of $u^*$ via the expansions
	\begin{align} \label{eq:hkstarorigdef}
		\fc_k^*&=\sum_{m=0}^\infty \fc_{k+m}\frac{{u^*}^m}{m!},\qquad k\ge1,
		\\C^*&=C-\fc_1^*,
		\\
    \exp\left(-\sum_{i=1}^\infty s^*_it^i\right)&=1-\sum_{k=1}^\infty\frac{\fc^*_{k+1}}{C-\fc^*_1}t^k.
    \label{eq:skorigdef}
	\end{align}
	
\end{theorem}

\begin{proof}
	Note that the classes $\Omega_{g,n}^{r,r}$ constitute a CohFT with the unit corresponding to the value of the index $a=r$ (cf. \cite[Lemma 2.1]{LPSZ}). In particular, they satisfy
	\begin{equation}
		\Omega_{g_0,n+p}^{r,r}(\ba,r^p)=\pi_{p}^*\Omega_{g_0,n}^{r,r}(\ba),
		\qquad\pi_{p}\colon\ocM_{g_0,n+p}\to\ocM_{g_0,n}.
	\end{equation}
	Therefore, by the projection formula, we get
	\begin{equation}\label{eq:proj}
		\int\limits_{\ocM_{g_0,n+p}}\frac{\Omega_{g_0,n+p}^{r,r}(\ba,r^p)\prod_{j=1}^p \Fc(\psi_{n+j})}{\prod_{i=1}^n(1-\frac{\mu_i}{r}\psi_i)}
		=
		\int\limits_{\ocM_{g_0,n}}\Omega_{g_0,n}^{r,r}(\ba)\;{\pi_{p}}_*\frac{\prod_{j=1}^p \Fc(\psi_{n+j})}{\prod_{i=1}^n(1-\frac{\mu_i}{r}\psi_i)}.
	\end{equation}
	In turn, the push-forward on the right hand side is computed in the following lemma proved in the next section.
	
	\begin{lemma}\label{lem:push}
		Let $\Fc(\psi)=\sum_{k=0}^\infty \fc_k\psi^k$ be a formal power series, $v_1,\dots,v_n$ be some constants, and $2g-2+n>0$. Then, the following identity holds in $H^*(\ocM_{g,n})$:
		\begin{equation}\label{eq:mainpush}
			\sum_{p=0}^\infty \frac{1}{p!}{\pi_{p}}_*\frac{\prod_{j=1}^p \Fc(\psi_{n+j})}{\prod_{i=1}^n(1-v_i\psi_i)}=
			(1-\fc_1^*)^{2-2g-n}e^{\bigl(\sum\limits_{i=1}^n v_i\bigr)u^*}\frac{e^{\sum_{k=1}^\infty s^*_k\kappa_k}}{\prod_{i=1}^n(1-v_i\psi_i)}.
		\end{equation}
		Here,  $u^*$ is a solution of the implicit functional equation
		\begin{equation}\label{eq:ustar}
			u^*=\sum_{m=0}^\infty \fc_m\frac{{u^*}^m}{m!},
		\end{equation}
		and $\fc_k^*$ and $s^*_k$ are expressed in terms of $u^*$ via the expansions
		\begin{align}
			\fc_k^*&=\sum_{m=0}^\infty \fc_{k+m}\frac{{u^*}^m}{m!},\qquad k\ge1,
			\\\exp\left(-\sum_{i=1}^\infty s^*_it^i\right)&=1-\sum_{k=1}^\infty\frac{\fc^*_{k+1}}{1-\fc^*_1}t^k.
		\end{align}
	\end{lemma}
	The formula of Proposition is obtained from~\eqref{eq:logJPT-preliminary} and~\eqref{eq:proj} by applying this lemma with $v_i=\mu_i/r$. To be precise, we use this lemma with two obvious modifications: first, we insert an additional rescaling parameter~$C$, and, second, we allow the dependence of the coefficients~$\fc_k$ of the series $\Fc(\psi)$ in~$\hbar$.
\end{proof}

\begin{remark} Given the methodology we developed here, it is fairly straightforward to write the formulas as in Theorem~\ref{prop:elsvformula} for $n=1$ and $n=2$ as well, they are just a bit more bulky, so we do not provide them here for brevity.
\end{remark}
	
\begin{remark}
	The formula of Lemma~\ref{lem:push} involves an infinite number of summands and the convergence of the formula should be justified. For example, we can treat Eq.~\eqref{eq:mainpush} in the formal expansion in $\fc_0,\fc_1$ considered as formal variables (while the dependence of the coefficient of any monomial in $\fc_0,\fc_1$ in $\fc_k$, $k\ge2$, is polynomial).
	
	The convergence of~\eqref{eq:mainformula} is justified, in turn, by the dependence of $\fc_k$ in $\hbar$ and the property $\fc_0\bigm|_{\hbar=0}=0$ and $\fc_1\bigm|_{\hbar=0}\ne C$. This property guarantees that the right hand side of~\eqref{eq:mainformula} has finitely many nonzero terms only.
\end{remark}

\begin{remark}	
	Note that Lemma \ref{lem:push} allows the following generalization: analogous arguments work if instead of the product $\prod\limits_{i=1}^n\frac 1{1-v_i\psi_i}$ we consider the product of $n$ formal power series $\Fc^{i}(\psi) = \sum\limits_{k=0}^\infty \fc_k^{(i)}\psi^k,\; i=1,\ldots,n.$	
\end{remark}

The following proposition provides an explicit formula for ${}^{\III}h^{(g)}_{\mu}$. Define the following functions depending on some variable $w$:
\begin{align}\label{eqB:uw}
u(w)&=\log(1-w)-r\,w,\quad \partial_u=\frac{w-1}{1+r-r \,w}\partial_w,
\qquad
\\W(w)&=W(w,\hbar)=\frac{1}{\cS(r\hbar\partial_u)}\log(1-w).
\end{align}

\begin{proposition}\label{prop:elsvformulazr}
The coefficients $u^*$, $C^*$, and $s^*_k$ in \eqref{eq:mainformula} are given by
\begin{align}
u^*&=\res_{w=0} u(w)\; d\log(W(w)),
\\C^*&=\res_{w=0}\partial_uW(w)\; d\log(W(w)),
\\\label{eqB:skw}
s^*_k&=
\res_{w=0}[\tau^k]\log\biggl(\frac{\partial_uW(w)}{\frac{\partial_u}{1-\tau\partial_u}W(w)}\biggr) d\log(W(w)).
\end{align}
\end{proposition}

\begin{proof}
	Follows directly from Theorem~\ref{prop:elsvformula} and formulas of Corollary~\ref{corB:res-formulas} of Appendix~\ref{sec:formulas} that we rewrote in the local coordinate $w=1+t=1+z^r$ on the spectral curve.
\end{proof}

\begin{remark}	
	The statement of Theorem~\ref{prop:elsvformula} does not in fact depend on the form of function $y(z)$. That is, one can generalize Theorem~\ref{prop:elsvformula} to the following. We assume that $y(z)$ is $\mathbb Z_r$-invariant, i.e.
	there exists a function $\widetilde y(t)$ such that $y(z)=\widetilde y(z^r).$
	We furthermore assume that $\widetilde y(t)$ is a sum of a meromorphic function and finitely many logarithmic terms, 
	\begin{equation}
		\widetilde y(t)=\widetilde y_{\mathrm{mer}}(t)+\sum_{j=1}^J \alpha_j \log g_j(t),
	\end{equation}
	where $\widetilde y_{\mathrm{mer}}(t)$ is meromorphic, $\alpha_j\in\mathbb C$, and $g_j(t)$ are meromorphic functions.
	In particular, the differential $dy$ is meromorphic on the $z$-sphere and has at worst simple poles at the logarithmic singularities of $y$. Finally, we assume that these logarithmic singularities are disjoint from the branch points of $x$ (so that all convolutions/residue computations will be well-defined).
\end{remark}

\subsection{Proof of Lemma~\ref{lem:push}}

In this section, we compute the push-forward homomorphism associated with the forgetful map $\pi_{p}\colon\ocM_{g,n+p}\to \ocM_{g,n}$ and prove Lemma~\ref{lem:push}.

First, consider the one-step projection $\pi\colon\ocM_{g,n+1}\to\ocM_{g,n}$. In order to compute $\pi_*P$ for some class $P$ expressed as a polynomial in $\psi$ and $\kappa$ classes we apply the substitution
\begin{equation}\label{eq:pikappa}
	\kappa_m=\pi^*\kappa_m+\psi_{n+1}^m.
\end{equation}
By the projection formula, this reduces the computation of $\pi_*$ to the case when $P$ is a monomial in $\psi$ classes only. Then, we apply the equality
\begin{equation}\label{eq:pipush}
	\pi_*\psi_1^{k_1}\dots\psi_n^{k_n}\psi_{n+1}^m=
	\left\{\begin{array}{ll}
		\sum\limits_{i:\;k_i>0}\psi_1^{k_1}\dots\psi_i^{k_i-1}\dots\psi_n^{k_n},&m=0,\\
		\psi_1^{k_1}\dots\psi_n^{k_n}\kappa_{m-1},&m\ge1.
	\end{array}\right.
\end{equation}
Here $\kappa_m=\pi_*\psi_{n+1}^{m+1}$. If $m>0$, it is the Morita--Mumford class, and for $m=0$ we have $\kappa_0=2g-2+n$. Note that~\eqref{eq:pikappa} is valid both for $m>0$ and $m=0$.

Applying repeatedly this inductive procedure we compute ${\pi_{p}}_*$ on any polynomial in $\psi$ and $\kappa$ classes again as a polynomial in $\psi$ and $\kappa$ classes (including the constant $\kappa_0$).
Let us denote the left hand side of~\eqref{eq:mainpush} by $\Psi(\fc_0,\fc_1,\fc_2,\dots)$ regarding $\fc_k$ as formal variables.
Our goal is to compute~$\Psi$ in a closed form.

Consider a special case $\fc_0=0$. Define $s_k,\, k\in\mathbb{Z}_{\geq 0}$, via the expansion
\begin{equation}
	\exp\left(-\sum_{k=0}^\infty s_kt^k\right)=1-\sum_{i=0}^\infty \fc_{i+1}t^i,
\end{equation}
or, equivalently,
\begin{equation}
	\exp\left(-\sum_{k=1}^\infty s_kt^k\right)=1-\sum_{i=1}^\infty \frac{\fc_{i+1}}{1-\fc_1}t^i,\quad \exp\left(-s_0\right)=1-\fc_1.
\end{equation}
Then, from~\eqref{eq:pipush} we get
\begin{equation}\label{eq:h0=0case}
	\begin{aligned}
		\Psi\bigm|_{\fc_0=0}&=\frac{1}{\prod_{i=1}^n(1-v_i\psi_i)}
		\sum_{p=0}^\infty \frac{1}{p!}{\pi_{p}}_*\prod_{j=1}^p \Fc(\psi_{n+j})\bigm|_{\fc_0=0}
		\\&=\frac{e^{\sum_{k=0}^\infty s_k\kappa_k}}{\prod_{i=1}^n(1-v_i\psi_i)} =(1-\fc_1)^{-\kappa_0}\frac{e^{\sum_{k=1}^\infty s_k\kappa_k}}{\prod_{i=1}^n(1-v_i\psi_i)}.
	\end{aligned}
\end{equation}

The general case can be reduced to the case $\fc_0=0$ by the following arguments. It follows from~\eqref{eq:pipush} that the $H^*(\ocM_{g,n})$-valued function $\Psi(\fc_0,\fc_1,\dots)$ satisfies the following relation called \emph{string equation}:
\begin{equation}
	\frac{\partial\Psi}{\partial \fc_0}-\sum_{k=0}^\infty \fc_{k+1}\frac{\partial\Psi}{\partial \fc_k}=
	\textstyle\left(\sum_{i=1}^n v_i\right)\Psi.
\end{equation}
Let us denote by $\fc_k(u)$ the phase trajectories of the vector field $\sum_{k=0}^\infty \fc_{k+1}\partial_{\fc_k}-\partial_{\fc_0}$:
\begin{equation}
	\fc_k(u)=\sum_{m=0}^\infty \fc_{k+m}\frac{u^m}{m!}-\delta_{k,0}u.
\end{equation}
Then, the string equation implies the identity
\begin{equation}
	\Psi\bigm|_{\fc_k\to \fc_k(u)}=e^{-(\sum_{i=1}^n v_i) u}\Psi
\end{equation}
that holds identically in~$u$. In particular, we can take the value $u=u^*$ determined by the equation $\fc_0(u^*)=0$ which is equivalent to~\eqref{eq:ustar}. Then, denoting $\fc_k^*=\fc_k(u^*)$ we get
\begin{equation}
	\Psi=e^{(\sum_{i=1}^n v_i) u^*}\Psi\bigm|_{\fc_0\to0,\;\fc_k\to \fc^*_k}.
\end{equation}
Substituting~\eqref{eq:h0=0case} we obtain the required formula of Lemma~\ref{lem:push}. This completes its proof.

\appendix

\section{Logarithmic projection property} \label{sec:logProj}

In this Appendix, we provide a direct proof of the projection property for the differentials of Definition~\ref{def:OSdifferentials} with the parameters determined by~\eqref{eq:fam3psi}--\eqref{eq:fam3xz} providing thereby an independent proof of Theorem~\ref{thm:fam3}. To begin with, we first recall an explicit formula for the differentials from~\cite[Proposition~2.3]{bychkov2021topological}, specialized to the present family.

\begin{proposition}[{Corollary of~\cite[Proposition~2.3]{bychkov2021topological}}]\label{prop:fam3explicitf}
	
	{\ }
	
	\begin{itemize}
		\item
		For $n\geq 2$, $(g,n)\neq(0,2)$ we have
		\begin{equation}\label{eq:mainprop}
			\frac{\omega^{(g)}_n(z_{\set{n}})}{\prod_{i=1}^n dx_i}=[\hbar^{2g-2+2n}]
			U_n\dots U_1
			\sum_{\gamma \in \Gamma_n}\prod_{\{v_k,v_\ell\}\in E_\gamma} w_{k,\ell},
		\end{equation}
		where the sum is over all connected simple graphs on $n$ labeled vertices,
		\begin{equation}\label{eq:wkl}
			w_{k,\ell}=e^{\hbar^2u_ku_\ell\cS(u_k\hbar\,z_k\partial_{z_k})\cS(u_\ell\hbar\,z_\ell \partial_{z_\ell})
				\frac{z_k z_\ell}{(z_k-z_\ell)^2}}-1
		\end{equation}
		and $U_{i}$ is the operator acting on a function~$f$ in~$u_i$ and $z_i$ by
		\begin{align}\label{eq:Uihbar}
			U_{i} f&=
			\sum_{j,r=0}^\infty \partial_{x_i}^j\bigg(\frac{d \log z_i}{d x_i}
			\left.\left([v^j]e^{-\alpha v\,e^{\theta}}\partial_\theta^r e^{\alpha v\cS(v\,\hbar)e^{\theta}}\right)\right|_{\theta=y(z_i)}
			[u_i^r] \frac{e^{u_i\left(\frac{\cS(u_i\,\hbar\,z_i\,\partial_{z_i})}{\cS(\hbar\,z_i\,\partial_{z_i})}-1\right)y(z_i)}}{u_i\,\cS(u_i\,\hbar)}f(u_i,z_i,\hbar)\bigg)
\\\notag			&=
			\sum_{j,r=0}^\infty \partial_{x_i}^j\bigg(\frac{d \log z_i}{d x_i}
			\left.\left([v^j]e^{-\alpha v\,t}(t\partial_t)^r e^{\alpha v\cS(v\,\hbar)t}\right)
\right|_{t=R(z_i)}
			[u_i^r] \frac{e^{u_i\left(\frac{\cS(u_i\,\hbar\,z_i\,\partial_{z_i})}{\cS(\hbar\,z_i\,\partial_{z_i})}-1\right)y(z_i)}}{u_i\,\cS(u_i\,\hbar)}f(u_i,z_i,\hbar)\bigg),
		\end{align}
		and $x_i$ stands for $x(z_i)$.
		
		\item	For $(g,n)=(0,2)$ we have
		$\omega^{(0)}_2 =\frac{dz_1dz_2}{(z_1-z_2)^2}$.
		
		\item	 For $n=1$, $g>0$ we have
		\begin{align}\label{eq:Wg1}
			\frac{\omega^{(g)}_1(z_1)}{dx_1}&= [\hbar^{2g}] \,U_1 1 + [\hbar^{2g}]\sum_{j=0}^\infty \partial_{x_1}^{j}\Big(\left([v^{j+1}] e^{\alpha v\left(\cS(v\,\hbar)-1\right)e^\theta}\right)\Bigm|_{\theta=y(z_1)}\; \partial_{x_1}y(z_1)\Big)
\\ \notag			&= [\hbar^{2g}] \,U_1 1 + [\hbar^{2g}]\sum_{j=0}^\infty \partial_{x_1}^{j}\Big([v^{j+1}] e^{\alpha v\left(\cS(v\,\hbar)-1\right)R(z_1)}\; \partial_{x_1}y(z_1)\Big).
		\end{align}
		
		\item	Finally, for $(g,n)=(0,1)$ we have
		$\omega^{(0)}_1 = y(z_1)dx(z_1)$.
	\end{itemize}		
\end{proposition}

For $(g,n)\neq(0,2)$ the differentials $\omega_n^{(g)}$ have no diagonal poles~\cite{BDKS1,bychkov2021topological}. Hence, for $2g-2+n>0$, the only possible poles are located at the zeros of $dx$, the zeros and poles of $R(z)$, and at $0$ and $\infty$. The LogTR-vital singularities are precisely the zeros of $R(z)$; throughout the proof we assume that these zeros are simple. Therefore it remains to show that
\begin{itemize}
	\item $\omega_n^{(g)}$ has no poles at the poles of $R(z)$, at $0$, and at $\infty$;
	\item for $n\ge2$ it is regular at the zeros of $R(z)$;
	\item for $n=1$ its principal part at a zero of $R(z)$ is the one prescribed by~\eqref{eq:PrincipalPartsLog}.
\end{itemize}

To lighten the notation we set $\alpha=1$; the general case is recovered by a trivial rescaling.

\subsection{Technical lemmas refining the ones from\texorpdfstring{~\cite{bychkov2021topological}}{ BDBKS21}}

We start with a sharpening of~\cite[Lemma~4.1]{bychkov2021topological}.

\begin{lemma}\label{lem:ulog}
	Let $u$ and $w$ be formal variables. Then for every $k\in \mathbb{Z}_{\geq 0}$,
	\begin{equation}\label{eq:ulog}
		[\hbar^{2k}]e^{u\left(\frac{\cS(u\,\hbar\,\partial_w)}{\cS(\hbar\,\partial_w)}-1\right)\log w } = \dfrac{u(u-1)\cdots (u-2k+1)}{w^{2k}}[\hbar^{2k}]\cS(\hbar)^{-u-1}.
	\end{equation}	
\end{lemma}
\begin{proof}
	Both sides are polynomials in $u$, so it is enough to check the identity for $u=-m-1$, $m\in\mathbb Z_{\ge0}$. For such $u$ we have
	\begin{align}
\notag		e^{u\left(\frac{\cS(u\hbar\partial_w)}{\cS(\hbar\partial_w)}-1\right)\log w}
		&=\exp\!\Big(\Big(-\frac{e^{\frac{m+1}{2}\hbar\partial_w}-e^{-\frac{m+1}{2}\hbar\partial_w}}{e^{\frac12\hbar\partial_w}-e^{-\frac12\hbar\partial_w}}+m+1\Big)\log w\Big) \\
\notag
		&=\frac{w^{m+1}}{\prod_{i=0}^m\bigl(w-(i-\frac m2)\hbar\bigr)}
\\\label{eq:ulogpoof1}
&=\sum_{i=0}^m\frac{(-1)^{m-i}(i-\frac m2)^m}{i!(m-i)!}\frac{1}{1-(i-\frac m2)\hbar w^{-1}}.
	\end{align}
In the final step, we perform a partial fraction decomposition. The coefficients of this decomposition are computed by taking residues at the corresponding poles.

For the function on the other side of the equality we have similarly
\begin{align}\label{eq:ulogpoof2}
\cS(\hbar)^{-u-1}&=\hbar^{-m}\bigl(e^{\hbar/2}-e^{-\hbar/2}\bigr)^m
=\hbar^{-m}\sum_{i=0}^m\frac{(-1)^{m-i}m!}{i!(m-i)!}e^{(i-\frac{m}{2})\hbar}.
\end{align}
The assertion of Lemma is obtained by comparing the coefficients of $\hbar^{2k}$ in the $i$th terms of the sums on the right hand sides of~\eqref{eq:ulogpoof1} and~\eqref{eq:ulogpoof2} for the same values of~$i$: they differ by an independent of~$i$ common factor
\begin{equation}
\frac{(2k+m)!}{m!}w^{-2k}=\dfrac{u(u-1)\cdots (u-2k+1)}{w^{2k}},
\end{equation}
and this is exactly what the lemma claims.
\end{proof}

\begin{corollary}\label{cor:zdzulog}	
	Let $A\in \mathbb{C}$ be fixed, and let $u$ and $z$ be formal variables. Then for every $k\in \mathbb{Z}_{\geq 0}$,
	\begin{equation}\label{eq:zdzulog}
		[\hbar^{2k}]e^{u\left(\frac{\cS(u\,\hbar\,z\partial_z)}{\cS(\hbar\,z\partial_z)}-1\right)\log\left(\log z-\log A\right) } = \dfrac{u(u-1)\cdots (u-2k+1)}{(\log z- \log A)^{2k}}[\hbar^{2k}]\cS(\hbar)^{-u-1}.
	\end{equation}	
\end{corollary}
\begin{proof}
	Apply Lemma~\ref{lem:ulog} with $w=\log z-\log A$, so that $\partial_w=z\partial_z$.
\end{proof}

\subsection{Zeros of \texorpdfstring{$R(z)$}{R(z)}}

We begin with the behavior at a zero of $R(z)$.

\begin{lemma}\label{lem:UzerosR}
	Let $A$ be a zero of $R(z)$. Let $f(u,z,\hbar)$ be a formal power series in $\hbar$ whose coefficients are polynomial in $u$ and regular in $z$ at $A$. Then, with $U$ denoting one of the operators $U_i$ and $x=x(z)$,
	\begin{equation}
		[\hbar^{2m}]U\bigl(u f(u,z,\hbar)\bigr)
	\end{equation}
	is regular at $z=A$ for every $m\ge0$. Moreover, for every $g\ge1$,
	\begin{equation}
		[\hbar^{2g}]U(1)=\bigl[\hbar^{2g}\bigr]\frac{1}{\cS(\hbar\partial_x)}\log(z-A)+(\mathrm{holomorphic}).
	\end{equation}
\end{lemma}

\begin{proof}
Notice that the function $x$ is holomorphic at $z=A$ and $dx$ is nonvanishing. It follows than both $x(z)-x(A)$ and $w(z)=\log z-\log A$ can be taken as local coordinates at $A$. Write
	\begin{equation}
		y(z)=\log w(z)+\eta_A(z),
	\end{equation}
	where $\eta_A=\log\bigl(R(z)/w(z)\bigr)$ is regular at $A$. Substituting this into~\eqref{eq:Uihbar} and using Corollary~\ref{cor:zdzulog}, we obtain
	\begin{align}\label{eq:fam3Uexpr1-new}
		Uf&=
		\sum_{j,r,k\ge0}\hbar^{2k}\partial_x^j[v^j]\Bigg(\frac{dw}{dx}\frac{1}{w^{2k}}
		\left.\Big(e^{-v t}(t\partial_t)^r e^{v\cS(v\hbar)t}\Big)\right|_{t=R(z)}
		\notag\\
		&\qquad\qquad\times [u^r]u(u-1)\cdots(u-2k+1)
		\Big(
		\frac{E_A(u,z,\hbar)}{u}f(u,z,\hbar)[\hbar^{2k}]\cS(\hbar)^{-u-1}\Bigg),
	\end{align}
	where
	\begin{equation}
		E_A(u,z,\hbar)\coloneqq
		e^{u\left(\frac{\cS(u\hbar z\partial_z)}{\cS(\hbar z\partial_z)}-1\right)\eta_A(z)}\frac{1}{\cS(u\hbar)}
	\end{equation}
	has coefficients polynomial in $u$ and regular in $z$ at $A$.

	For the first claim we replace $f$ by $uf$; this cancels the factor $1/u$ in~\eqref{eq:fam3Uexpr1-new}. For each $r\ge0$, the function	$e^{-vt}(t\partial_t)^r e^{v\cS(v\hbar)t}$ is a series in $\hbar$ with coefficients polynomial in~$t$ and $v$. On the other hand, for a fixed $k\ge0$, the falling factorial in $u$ turns into the differential operator
	\begin{equation}\label{eq:t-falling}
		t\partial_t(t\partial_t-1)\cdots(t\partial_t-2k+1)
	\end{equation}
that kills all monomials in~$t$ of degrees smaller than $2k$. Substituting $t=R(z)$ we obtain a function with zero of order at least~$2k$ at~$A$. This zero compensates the pole of $\frac{1}{w^{2k}}$ so that the function to which the operator $\partial_x^j[v^j]$ is applied is regular at $z=A$. This proves the first claim.
	
	Now consider $U(1)$. In this case, we have no factor~$u$ in~\eqref{eq:fam3Uexpr1-new} any more corresponding to the factor~$t\partial_t$ in~\eqref{eq:t-falling}. However, exactly the same argument shows that any term containing either
	\begin{itemize}
		\item a positive power of $u$ from $E_A(u,z,\hbar)$, or
		\item a positive $\hbar$-correction from $e^{v\cS(v\hbar)t}=e^{vt}(1+\hbar^2t\,\mathrm{reg})$
	\end{itemize}
	is holomorphic at $A$: the first case reduces to the already proved regularity of $U(uf)$, while in the second case the factor~$t\partial_t$ killing constants is not needed since we apply it to a series in~$t$ with vanishing constant term. Therefore, the singular part comes only from replacing $E_A$ by $1$, replacing $e^{v\cS(v\hbar)t}$ by $e^{vt}$, and taking $k=g$. We obtain
	\begin{align}
		[\hbar^{2g}]U(1)
		&=c_g\sum_{j\ge0}\partial_x^j[v^j]\Bigg(\frac{dw}{dx}\frac{1}{w^{2g}}
		\left.\Big(e^{-vt}
		(t\partial_t-1)\cdots(t\partial_t-2g+1)e^{vt}\Big)\right|_{t=R(z)}\Bigg)
		+(\mathrm{holomorphic}),
	\end{align}
	where $c_g\coloneqq[\hbar^{2g}]\,\cS(\hbar)^{-1}$. Since
	\begin{equation}
		e^{-vt}(t\partial_t-1)\cdots(t\partial_t-2g+1)e^{vt}
		=\sum_{j=0}^{2g-1}(-1)^{j+1}\frac{(2g-1)!}{j!}v^jt^j
	\end{equation}
	we obtain
	\begin{equation}\label{eq:U1-singular-form}
		[\hbar^{2g}]U(1)
		=c_g\sum_{j=0}^{2g-1}(-1)^{j+1}\frac{(2g-1)!}{j!}
		\partial_x^j\left(\frac{dw}{dx}\frac{R(z)^j}{w^{2g}}\right)
		+(\mathrm{holomorphic}).
	\end{equation}
Comparing with the required claim we see that the second assertion of Lemma follows from the equality
\begin{equation}\label{eq:lemR0proof-ident}
\sum_{j=0}^{m-1}(-1)^{m-1-j}\frac{(m-1)!}{j!}
		\partial_x^j\left(\frac{dw}{dx}\frac{R(z)^j}{w^{m}}\right)=
\partial_x^{m}\log(z-A)+(\mathrm{holomorphic})
\end{equation}
that holds for any integer $m\ge1$. In fact, we need this equality for $m=2g$ but let us prove it for any positive~$m$.

Set $\tilde x(z)=x(z)-x(A)$. Then, we have $R(z)=w(z)-\tilde x(z)$ and the left hand side of~\eqref{eq:lemR0proof-ident} can be written as
\begin{equation}\label{eq:U1contributions}
\begin{aligned}
\sum_{j=0}^{m-1}&(-1)^{m-1-j}\frac{(m-1)!}{j!}
		\partial_{\tilde x}^j\left(\frac{dw}{d\tilde x}\frac{(w-\tilde x)^j}{w^{m}}\right)
\\&=\sum_{i+\ell\le m-1}(-1)^{m-1-i}\frac{(m-1)!}{i!\ell!}
		\partial_{\tilde x}^{i+\ell}\left(\frac{dw}{d\tilde x}w^{i-m}\tilde x^\ell\right)
\\&=\sum_{k=0}^{m-1}\sum_{\ell=0}^k(-1)^{k}\frac{(m-1)!}{(m-k-1)!\ell!}
		\partial_{\tilde x}^{m-k+\ell-1}\left(\frac{dw}{d\tilde x}w^{-k-1}\tilde x^\ell\right).
\end{aligned}
\end{equation}
The term of this sum with $k=\ell=0$ is equal to
\begin{equation}
		\partial_{\tilde x}^{m-1}\left(\frac{dw}{d\tilde x}w^{-1}\right)
=\partial_{\tilde x}^{m}\log w=\partial_{\tilde x}^{m}\log (z-A)
+(\mathrm{holomorphic}),
\end{equation}
that is, it coincides with the right hand side of~\eqref{eq:lemR0proof-ident}. For the remaining terms of~\eqref{eq:U1contributions} we use the Leibniz rule $\partial_{\tilde x} f\;g=\partial_{\tilde x}(f\,g)-f\partial_{\tilde x} g$ to represent them as
\begin{multline}\label{eq:U1contrib2}
\sum_{k=1}^{m-1}\sum_{\ell=0}^k(-1)^{k}\frac{(m-1)!}{(-k)(m-k-1)!\ell!}
		\partial_{\tilde x}^{m-k+\ell-1}\left(\partial_{\tilde x}w^{-k}\;\tilde x^\ell\right)
\\=
\sum_{k=1}^{m-1}\sum_{\ell=0}^k(-1)^{k}\frac{(m-1)!}{(-k)(m-k-1)!\ell!}
		\partial_{\tilde x}^{m-k+\ell}\left(w^{-k}\tilde x^\ell\right)
\\-\sum_{k=1}^{m-1}\sum_{\ell=1}^k(-1)^{k}\frac{(m-1)!}{(-k)(m-k-1)!(\ell-1)!}
		\partial_{\tilde x}^{m-k+\ell-1}\left(w^{-k}\;\tilde x^{\ell-1}\right).
\end{multline}
The terms of the first sum with $\ell\le k-1$ are cancelled by the corresponding terms of the second summand. Finally, the terms of the first sum with $\ell=k$ are in fact holomorphic. We conclude that the whole sum~\eqref{eq:U1contrib2} is holomorphic. This completes the proof of~\eqref{eq:lemR0proof-ident} that implies, in turn, the second assertion of Lemma.
\end{proof}

The remaining term in~\eqref{eq:Wg1} is easier.

\begin{lemma}\label{lem:w1sectermzerosR}
	Let $A$ be a zero of $R(z)$. For every $g\ge0$, the expression
	\begin{equation}\label{eq:fam3w1secondterm}
		[\hbar^{2g}]\sum_{j=0}^\infty \partial_{x_1}^{j}\left([v^{j+1}] e^{ v\left(\cS(v\,\hbar)-1\right)R(z_1)}\; \partial_{x_1}y(z_1)\right)
	\end{equation}
	is regular at $z_1=A$.
\end{lemma}
\begin{proof}
The coefficient of any positive power of~$v$ in $e^{ v\left(\cS(v\,\hbar)-1\right)R(z_1)}$ has necessarily a factor $R(z_1)$, and $R(z_1)\partial_{x_1}y(z_1)$ is regular at $A$. Hence~\eqref{eq:fam3w1secondterm} is regular at $A$.
\end{proof}

\subsection{Poles of \texorpdfstring{$R(z)$}{R(z)}}

We now analyze the behavior at a pole of $R(z)$.

\begin{lemma}\label{lem:fam3polesR}
	Let $B$ be a pole of $R(z)$. For $2g-2+n>0$, the differential $\omega_n^{(g)}$ is regular in $z_1$ at $z_1=B$.
\end{lemma}
\begin{proof}
	Let $U$ denote the operator~\eqref{eq:Uihbar} with the index $i$ suppressed. Since $R$ has a simple pole at $B$, formula~\eqref{eq:fam3xz} shows that $x$ also has a simple pole there and $dx$ has a pole of order~$2$.
Using Corollary~\ref{cor:zdzulog} with $u\mapsto -u$, the same rearrangement as in the proof of Lemma~\ref{lem:UzerosR} shows that every coefficient $[\hbar^{2m}]Uf$ is a finite linear combination of terms obtained from
	\begin{multline}\label{eq:poles-master}
		\sum_{j\ge0,k \geq 1}\hbar^{2k}\partial_x^j\Bigg(\rho_{2k-2}(z)
		\left.\Big([v^j]\widetilde f(t\partial_t+tv,z,\hbar)
		(t\partial_t+tv+1)\cdots(t\partial_t+tv+2k-1)e^{v(\cS(v\hbar)-1)t}\Big)\right|_{t=R(z)}\Bigg)\\
		+\sum_{j\ge0}\partial_x^j\Bigg(\rho_{-2}(z)\left.\Big([v^j]\widetilde f(t\partial_t+tv,z,\hbar)
		e^{v(\cS(v\hbar)-1)t}\Big)\right|_{t=R(z)}\Bigg),
\end{multline}
	where $\widetilde f$ is regular in $z$ at~$B$ and polynomial in its first argument while $\rho_{2k-2}(z)$ has a pole of order at most~$2k-2$, and $\rho_{-2}(z)$ has a zero of order at least 2.

After expanding the regular coefficients in \eqref{eq:poles-master}, it is enough to consider expressions
\begin{equation}
T_r\coloneqq\sum_{j\ge0}\partial_x^j[v^j]\bigl(\rho_r(z)\,D_{r+1}\cdots D_1\,v^p t^q\bigr)\Big|_{t=R(z)},
\qquad r\ge -1,\quad p\ge q,	
\end{equation}
where for $r=-1$ the empty product $D_{r+1}\cdots D_1$ is understood as $1$, $D_k\coloneqq t\partial_t+t v+k$, and
$\rho_r(z)$ has a pole of order at most $r$ for $r\ge0$, while $\rho_{-1}(z)$ has at least a
simple zero at $B$. 

We claim that every $T_r$ is holomorphic at $B$ and has at least a simple zero there.

For $r=-1$ we have
\begin{equation}
T_{-1}=\partial_x^p\bigl(\rho_{-1}(z)R(z)^q\bigr).
\end{equation}
Since $R$ has a simple pole at $B$ and $\partial_x$ raises the order at $B$ by $1$ (because $dx$
has a double pole there), while $\rho_{-1}$ has at least a simple zero, the condition $p\ge q$
implies that $T_{-1}$ has at least a simple zero at $B$.

For $r\ge0$ we use the following identity: for any function $F(z,t,v)$ polynomial in~$z$ and~$v$ we have
\begin{multline}
\sum_{j\ge0}\partial_x^j[v^j]\bigl(t\partial_t F(z,t,v)\bigr)\Bigm|_{t=R(z)}
\\=\sum_{j\ge0}\partial_x^j[v^j]\left(
\Bigl(\partial_x+\frac{dR}{dx}\partial_t\Bigr)
R\frac{dx}{dR}F-\frac{\partial\bigl(R\tfrac{dx}{dR}F\bigr)}{\partial x}
\right)\Bigm|_{t=R(z)}
\\=\sum_{j\ge0}\partial_x^j[v^j]\left(
t\,v\,\frac{dx}{dR}F-\frac{\partial\bigl(R\tfrac{dx}{dR}F\bigr)}{\partial x}
\right)\Bigm|_{t=R(z)}.
\end{multline}
From this identity, denoting $W_{r-1}(t,v)=D_{r}\dots D_{1}v^pt^q$ we compute:
\begin{multline}
T_r=\sum_{j\ge0}\partial_x^j[v^j]\Bigl(\rho_r(z)(t\partial_t+v t+(r+1))W_{r-1}(v,t)\Bigr)\Bigm|_{t=R(z)}
\\=
\sum_{j\ge0}\partial_x^j[v^j]\Bigl(\Bigl(t\,v\,\rho_r(z)\left(\tfrac{dx}{dR}+1\right)+
\bigl((r+1)\rho_r(z)-\frac{d(R\tfrac{dx}{dR}\rho_r(z))}{dx}
\Bigr)W_{r-1}(t,v)\Bigr)\Bigm|_{t=R(z)}.
\end{multline}
Let us show that the factor in front of $W_{r-1}$ has a pole of order at most $r-1$. Indeed, we have $x(z)=-R(z)+\text{(reg)}$, so $\tfrac{dx}{dz} + \tfrac{dR}{dz} = O(1)$, but both $\tfrac{dx}{dz}$ and $\tfrac{dR}{dz}$ have double poles at $B$. Therefore $\tfrac{dx}{dR}+1=\tfrac{\tfrac{dx}{dz} + \tfrac{dR}{dz}}{\tfrac{dR}{dz}} = O((z-B)^2)$ and the whole summand $tv\rho_r(z)\bigl(\tfrac{dx}{dR}+1)$ has a pole of order at most~$r-1$.

By a direct local computation, one can see that the leading terms of poles (of order~$r$) of the two summands in $(r+1)\rho_r(z)-\frac{d\bigl(R\tfrac{dx}{dR}\rho_r(z)\bigr)}{dx}$ cancel each other and the result has a pole of order at most~$r-1$.

We obtain that any function of type $T_r$ can be rewritten as a function of type $T_{r-1}$. By induction, we conclude that $T_r$ is holomorphic with at least simple zero.
	
	Thus $[\hbar^{2m}]Uf$ has at least a simple zero at $B$ whenever $f$ is regular in $z$ at $B$ and polynomial in $u$. Note that each product $\prod_{\{v_k,v_\ell\}\in E_\gamma}w_{k,\ell}$ in~\eqref{eq:mainprop} is of this form.

	Now consider the second term in~\eqref{eq:Wg1}. Since $\cS(v\hbar)-1=O(v^2)$, its coefficient of $\hbar^{2g}$ is a finite linear combination of terms of the form
	\begin{equation}
		\partial_x^j\bigl(R(z)^i\partial_xy(z)\bigr),\qquad i\le j.
	\end{equation}
	The same pole-counting argument shows that these terms have at least a simple zero at $B$ as well.
	
	What remains is to note that after multiplication by $dx_1\cdots dx_n$ the resulting $\omega^{(g)}_n$ is holomorphic. Consider a given variable $z_i$. Since $dx(z_i)$ has a pole of order 2 at $z_i=B$, after the multiplication by $dx(z_i)$ the resulting expression can have at most a simple pole in $z_i$ at $z_i=B$. But~\cite[Theorem~5.3]{BDKS1} implies that $\omega^{(g)}_n=d_1\dots d_n H^{(g)}_{n}(z_1,\dots, z_n)$ where $H^{(g)}_{n}$ is a meromorphic function, and the differential of a meromorphic function can never have a simple pole. Thus  $\omega^{(g)}_n$ is holomorphic.
	This proves the lemma.

\end{proof}

\begin{remark}
	In Lemmas~\ref{lem:UzerosR} and~\ref{lem:fam3polesR} it is essential that we use precisely the $\hbar$-deformation~\eqref{eq:fam3y}. Otherwise Corollary~\ref{cor:zdzulog} would not apply, and the cancellations at the zeros and poles of $R$ would be lost.
\end{remark}

\subsection{Zero and infinity}

\begin{lemma}\label{lem:fam3zero}
	For $2g-2+n>0$, the differential $\omega_n^{(g)}$ is regular in $z_1$ at $z_1=0$.
\end{lemma}
\begin{proof}
	The expressions $\omega_n^{(g)} / \prod dx_i$ given by~\eqref{eq:mainprop} and~\eqref{eq:Wg1} are regular at $z_1=0$. Indeed, the factors $w_{k,\ell}$ are regular at $z_1=0$ and $y(z_1)=\log R(z_1)$ is regular at $z_1=0$ because $R(0)=1$. The prefactor in~\eqref{eq:Uihbar} is
	\begin{align}
		\frac{d\log z_1}{dx_1} &= \frac{1}{1-z_1R'(z_1)},
	\end{align}
	which is regular at $z_1=0$. The differential operator
	\begin{align}
		\partial_{x_1} &= \frac{z_1}{1-z_1R'(z_1)}\partial_{z_1}
	\end{align}
	preserves regularity. 
	
	To obtain $\omega^{(g)}_n$, we multiply by $dx_1\cdots dx_n$. We have
	\begin{align}
		dx_1 &= \frac{1-z_1R'(z_1)}{z_1}dz_1.
	\end{align}
	Multiplication by $dx_1$ produces at worst a simple pole at $z_1=0$. As noted in the proof of Lemma~\ref{lem:fam3polesR}, the differentials $\omega^{(g)}_n$ cannot have simple poles. Therefore, $\omega^{(g)}_n$ is regular at $z_1=0$.
\end{proof}

\begin{lemma}\label{lem:fam3infty}
	For $2g-2+n>0$, the differential $\omega_n^{(g)}$ is regular in $z_1$ at $z_1=\infty$.
\end{lemma}
\begin{proof}
	Write $z=z_1$. If $R(z)$ has a finite limit at $\infty$, then after the change of variable $t=1/z$ the argument is exactly the same as in Lemma~\ref{lem:fam3zero} (if $R(\infty)=0$, the resulting logarithmic singularity of $y$ at $t=0$ is annihilated by the operator $\frac{\cS(u\hbar t\partial_t)}{\cS(\hbar t\partial_t)}-1$ in~\eqref{eq:Uihbar}, which consists of positive even powers of $t\partial_t$). We therefore assume that
	\begin{equation}
		R(z)=\frac{P_1(z)}{P_2(z)},\qquad \Delta\coloneqq\deg P_1-\deg P_2>0.
	\end{equation}
	Then
	\begin{equation}
		R(z) = O(z^\Delta),
		\qquad
		\frac{d\log z}{dx} =  O(z^{-\Delta}),
		\qquad
		\frac{dz}{dx} = O(z^{-\Delta+1})
		\qquad (z\to\infty).
	\end{equation}
	
	The product $\prod w_{k,\ell}$ is regular at $\infty$, and for every regular $f$ the factor
	\begin{equation}
		[\hbar^{2m}][u^r]\frac{e^{u(\frac{\cS(u\hbar z\partial_z)}{\cS(\hbar z\partial_z)}-1)y(z)}}{u\cS(u\hbar)}f(u,z,\hbar)
	\end{equation}
	is also regular at $\infty$, because the operator in the exponent contains at least one $\partial_z$ acting on $y(z)=\log R(z)$. Furthermore,
	\begin{equation}
		[\hbar^{2m}]\Bigl(e^{-vt}(t\partial_t)^re^{v\cS(v\hbar)t}\Bigr)\Bigm|_{t=R(z)}
	\end{equation}
	is a finite linear combination of monomials $v^\ell R(z)^i$ with $0\le i\le \ell$. Hence every coefficient $[\hbar^{2m}]Uf$ is a finite linear combination of terms of the form
	\begin{equation}
		\left(\frac{dz}{dx}\partial_z\right)^\ell\left(\frac{d\log z}{dx}R(z)^i\,\mathrm{reg}(z)\right),
		\qquad 0\le i\le \ell,
	\end{equation}
	where $\mathrm{reg}(z)$ is regular at $\infty$. Since
	\begin{equation}
		\frac{d\log z}{dx}R(z)^i=O(z^{(i-1)\Delta}),
	\end{equation}
	after applying $\left(\frac{dz}{dx}\partial_z\right)^{i-1}$ we obtain $c_0+c_1/z+O(z^{-2})$. If $\ell > 0$, one more application yields $O(z^{-\Delta-1})$. If $\ell=0$ (which implies $i=0$), the term is $O(z^{-\Delta})$. Therefore every such term is at worst $O(z^{-\Delta})$, and after multiplication by
	\begin{equation}
		dx = O(z^{\Delta-1})dz
	\end{equation}
	we get at worst $O(z^{-1})dz$, which corresponds to at worst a simple pole at $z=\infty$. As explained at the end of the proof of Lemma~\ref{lem:fam3polesR}, the differentials $\omega_n^{(g)}$ cannot have simple poles, so all possible simple poles must cancel in the end.
	
	For the second summand in~\eqref{eq:Wg1}, note that $\cS(v\hbar)-1=O(v^2)$. Hence the coefficient of $R(z)^i$ in
	\begin{equation}
		e^{v(\cS(v\hbar)-1)R(z)}
	\end{equation}
	contains at least $v^{3i}$. After extracting $[v^{j+1}]$, this yields terms of the form
	\begin{equation}
		\left(\frac{dz}{dx}\partial_z\right)^\ell\bigl(R(z)^i\partial_xy(z)\bigr),
		\qquad \ell\ge 3i-1,
	\end{equation}
	while $\partial_xy(z)=\frac{dz}{dx}\partial_z\log R(z)=O(z^{-\Delta})$. The same power counting therefore gives again $O(z^{-\Delta-1})$, and the product with $dx$ is regular at $\infty$. This proves the lemma.
\end{proof}

\begin{remark}
	At infinity it is essential that $\hat\psi$ is the deformation~\eqref{eq:fam3psi}; otherwise the $j=0$ term in the second summand of~\eqref{eq:Wg1} would produce a pole.
\end{remark}

\subsection{Proof of LogTR for Family~III}

\begin{proof}[Proof of Theorem~\ref{thm:fam3}]
	The loop equations hold by Proposition~\ref{prop:fam3loopeq}, so it remains to check the logarithmic projection property.
	
	By Lemmas~\ref{lem:fam3polesR}, \ref{lem:fam3zero}, and~\ref{lem:fam3infty}, the differentials $\omega_n^{(g)}$ have no poles at the poles of $R(z)$, at $0$, or at $\infty$. If $n\ge2$, then every term in~\eqref{eq:mainprop} is divisible by $u_1\cdots u_n$; hence the first part of Lemma~\ref{lem:UzerosR} applies and shows that $\omega_n^{(g)}$ is regular at the zeros of $R(z)$.
	
	For $n=1$, Lemma~\ref{lem:UzerosR} gives the singular part of the first term in~\eqref{eq:Wg1}, while Lemma~\ref{lem:w1sectermzerosR} shows that the second term is regular at the zeros of $R(z)$. Therefore the principal part of $\omega_1^{(g)}$ at every zero of $R(z)$ is exactly the one prescribed by~\eqref{eq:PrincipalPartsLog}. This proves the logarithmic projection property and completes the proof of Theorem~\ref{thm:fam3}.
\end{proof}

\section{Closed formulas for coefficients} \label{sec:formulas}
This section is devoted to computing the coefficients $\fc_k$, $u^*$, $\fc_k^*$, and $s_k^*$ of Theorem~\ref{prop:elsvformula} for the specific case of spectral curve~\eqref{eq:curvezr}, in order to derive Proposition~\ref{prop:elsvformulazr}.

All the functions of the spectral curve data~\eqref{eq:curvezr} are invariant under the action of the group $\Z_r$ of $r$-th roots of unity. Therefore, they can be expressed as single-valued functions of $t=z^r$. Moreover, all the residues entering the formulas for the coefficients are the same for all poles with common value of~$t$. It follows that they can be replaced by just a single residue of a differential on the $t$-line multiplied by the factor~$r$. In other words, all computations can be done entirely in the coordinate~$t=z^r$.
Let us summarize the definitions that we need to perform computations in terms of the coordinate $t$.

The equation of the spectral curve \eqref{eq:curvezr} in terms of~$t$ becomes
\begin{equation}
x(t)=\frac{1}{r}\log t-t,
\quad y(t)=\log(1+t).
\end{equation}
The differential $dx(t)=\frac{1-r t}{rt}dt$ has a single zero at the point $t=1/r$ of the $t$-line and $y(t)$ has a logarithmic singularity at $t=-1$. The definitions of other objects related to this spectral curve take in terms of~$t$ the following form:
\begin{align}
\xi_k(t)&=\xi_{r,k}=r^{-k}\partial_x^{k}\frac{1}{1-r\,t},\qquad \partial_x=\frac{r t}{1-r t}\partial_t,
\\C&=\frac{1}{r},
\\[-2ex]\Phi(t)&=(y-t)\,dx+\left(\res_{\tilde t=-1}\frac{\left(\left(\frac{1}{\cS(\hbar\partial_{x(\tilde t)})}-1\right)y(\tilde t)\right)
dx(\tilde t)}{t-\tilde t}\right)dt,
\\\label{eqB:hk0}
\fc_k&=d\xi_k * \Phi=-r\res_{t=1/r}\xi_k(t)\,\Phi(t)
.
\end{align}

Recall that $\Phi,\fc_k$ are expanded as power series in~$\hbar^2$, and for brevity we drop the dependence on~$\hbar$ from notations. A more explicit computation of these coefficients is provided by the following Proposition. Consider the following function (a formal power series)
\begin{equation}\label{eqB:Wu}
W(u,\hbar)=\frac{u}{r}-\sum_{m=0}^\infty \fc_m\frac{u^m}{m!}.
\end{equation}

Let us identify the variable $u$ with the local coordinate on the spectral curve near the point $t=-1$ defined by
\begin{equation}
u(t)=r\;(x(t)-x(-1))=\log(-t)-r(t+1)
\end{equation}
With the identification, $W$ becomes a function on the spectral curve (expanded in~$\hbar^2$). 
 
\begin{proposition}\label{prop:W-series}
The function~$W$ regarded as a function on the spectral curve is given by
\begin{equation}\label{eqB:Wt}
W(u,\hbar)=\frac{1}{\cS(\hbar\partial_x)}\log(-t).
\end{equation}
\end{proposition}

Notice that the function $\log(-t)$ is well defined and regular at the point $t=-1$. Using that
$
\partial_x\log(-t)=\frac{r}{1-rt}=r\,\xi_0,%
$
we obtain:
\begin{equation}
\begin{aligned}
W&=\log(-t)-\frac{r^2}{24}\xi_1\hbar^2+\frac{7\,r^4}{5760}\xi_3\hbar^4+\dots
\\&=\log(-t)-\frac{r^2}{24}\frac{r t}{(1-r t)^3}\hbar^2+\frac{7\,r^4}{5760}\frac{r t \left(1+8 r t+6 r^2 t^2\right)}{(1-r t)^7}\hbar^4+\dots
\end{aligned}
\end{equation}

\begin{proof}
For the local coordinate $u$, we have
\begin{equation}
\partial_u=r^{-1}\partial_x=\frac{t}{1-r t}\partial_t.
\end{equation}
Therefore, the Taylor coefficients of the function~\eqref{eqB:Wt} at the point $t=-1$ are given by
\begin{equation}
\partial_u^kW(u)\bigm|_{u=0}=\partial_{u(t)}^kW(u(t))\bigm|_{t=-1}=
\frac{r^{-k}\partial_{x}^k}{\cS(\hbar\partial_{x})}\log(-t)\bigm|_{t=-1}.
\end{equation}
Comparing with the Taylor expansion of~\eqref{eqB:Wu} we see that Proposition is equivalent to the following equality
\begin{equation}\label{eqB:hk2}
\fc_k=\frac{\delta_{k,1}}{r}-\frac{r^{-k}\partial_{x}^k}{\cS(\hbar\partial_{x})}\log(-t)\biggm|_{t=-1}.
\end{equation}
Let us prove it by evaluating explicitly the coefficients $\fc_k$ from their definition. We do it in two steps. In the first step, we introduce the differential
\begin{equation}
\Phi_0(t)=(y-t)\,dx+\left(\Bigl(\frac{1}{\cS(\hbar\partial_{x})}-1\Bigr)y\right)dx=\Bigl(-t+\frac{1}{\cS(\hbar\partial_x)}y\Bigr)dx.
\end{equation}
Compared with $\Phi$ given by \eqref{eq:Phi}, here we do not extract the principal part of the second summand. We claim, however, that we have
\begin{equation}
\label{eqB:hk1}
\fc_k=d\xi_k*\Phi_0=-r\res_{t=1/r}\xi_k(t)\,\Phi_0(t),
\end{equation}
Indeed, by construction, the differentials $\Phi_0$ and $\Phi$ have the same principal parts of the pole at $t=-1$ but $\Phi_0$ has an extra pole at $t=1/r$ produced by iterative applications of $\partial_x$. It follows that their difference $\Phi_0-\Phi$ has a unique pole at $t=1/r$. The product $\xi_k\,(\Phi_0-\Phi)$ also has a unique pole at $t=1/r$. Since this is a unique pole of a global meromorphic differential, its residue is necessarily vanishing. Hence, the residues~\eqref{eqB:hk0} and~\eqref{eqB:hk1} coincide.

Next, we apply repeatedly the equality $\res(\partial_xf\;g\;dx)=-\res(f\;\partial_xg\;dx)$ to the obtained expression for $\fc_k$. In the case $k\ge2$, ignoring the terms that do not contribute for $k\ge2$ and using that $\frac{1}{1-rt}=r^{-1}\partial_x\log(-t)$, we compute:
\begin{equation}
\begin{aligned}
\fc_k&=-r \res_{t=1/r}\xi_k\;
\left(\frac{1}{\cS(\hbar\partial_{x})}y\right)\;dx
=-\res_{t=1/r}r^{-k}\partial_x^{k}\partial_x\log(-t)\;
\left(\frac{1}{\cS(\hbar\partial_{x})}y\right)\;dx
\\&
=\res_{t=1/r}r^{-k}\partial_x^{k}\log(-t)\;
\left(\frac{1}{\cS(\hbar\partial_{x})}\partial_xy\right)\;dx
=\res_{t=1/r}
\left(\frac{r^{-k}\partial_x^{k}}{\cS(\hbar\partial_{x})}\log(-t)\right)\;\partial_xy\,dx
\\&
=\res_{t=1/r}
\left(\frac{r^{-k}\partial_x^{k}}{\cS(\hbar\partial_{x})}\log(-t)\right)\;\frac{dt}{t+1}
=- \res_{t=-1}
\left(\frac{r^{-k}\partial_x^{k}}{\cS(\hbar\partial_{x})}\log(-t)\right)\;\frac{dt}{t+1}.
\end{aligned}
\end{equation}
This gives exactly~\eqref{eqB:hk2}. In the cases $k=0$ and $k=1$ the computations are similar, appropriately accounting for exceptional terms. This proves~\eqref{eqB:hk2}, and hence, Proposition.
\end{proof}

Next, we consider the implicit equation on the value $u^*$. With the introduced notations it reads
\begin{equation}\label{eqB:impliciteq}
W(u^*,\hbar)=0.
\end{equation}
A solution to this equation is a point $t^*=-1+O(\hbar^2)$, i.e. $u^*=O(\hbar^2)$, on the spectral curve in a vicinity of $t=-1$.
This equation is nondegenerate since $W(u,0)=\log(-t)$ has a nonzero derivative at the considered point. Therefore, its solution is uniquely defined as a series in~$\hbar$. For the local coordinate~$t$ it is given by
\begin{equation}
t^*=t(u^*)=-1+\frac{r^3 }{24\,(r+1)^3}\hbar ^2-\frac{r^5 \left(27 r^2-41 r+7\right) }{5760\,(r+1)^7}\hbar^4
+\dots
\end{equation}

There is a convenient way to find the solution explicitly using Lagrange inversion.

\begin{proposition}[Lagrange--B\"urmann inversion formula~\cite{Lagrange}]
\label{prop:LB}
For any power series $H(u)$ we have
\begin{equation}\label{eqB:L-B}
\begin{aligned}
H(u^*)&=H(0)-\res_{u=0}\log\bigl(W(u)/u\bigr)\,\partial_u H(u)\,du
\\&=\res_{u=0}H(u)\;\partial_u\!\log\bigl(W(u)\bigr)\,du.
\end{aligned}
\end{equation}
\end{proposition}

\begin{remark}
Here and below, we drop the dependence of~$W$ on~$\hbar$ from notations, for brevity.
We consider the differentials under the residues as formal power series in~$\hbar$. The coefficient of each monomial in~$\hbar$ has a pole in~$u$ of finite order and its residue is well defined.
\end{remark}

\begin{proof}%
	Essentially, the Lagrange--Bürmann formula is a corollary of invariance of residue of a meromorphic differential under a change of variables. Let us rewrite the implicit equation in the form $G(u)=-W(0)$ where $G(u)=W(u)-W(0)$. Since $G'(0)\neq 0$, we can regard $\epsilon=G(u)$ as a local change of variables. With $u(\epsilon)$ being the inverse change, the solution to the original equation is obtained by the substitution $u^*=u(-W(0))$. This substitution is well defined since $W(0)=O(\hbar)$ as a series in~$\hbar$.
	Thus, we need to compute the inverse function~$u(\epsilon)$ and also the composition $H(u(\epsilon))$ in the expansion in~$\epsilon$. We have:
	\begin{multline}
		[\epsilon^k]H(u(\epsilon))=\frac1{k}[\epsilon^{k-1}]\partial_\epsilon H(u(\epsilon))=
		\frac{1}{k}\res_{\epsilon=0}\frac{\partial_\epsilon H(u(\epsilon))d\epsilon}{\epsilon^k}
		\\=\frac{1}{k}\res_{u=0}\Bigl(\frac{1}{G(u)}\Bigr)^k \partial_u H(u)du
		=-[\epsilon^k]\res_{u=0}\log\Bigl(1-\frac{\epsilon}{G(u)}\Bigr) \partial_u H(u)du.
	\end{multline}
	This leads to the equality
	\begin{equation}
		H(u(\epsilon))=-\res_{u=0}\log\Bigl(\frac{G(u)-\epsilon}{G(u)}\Bigr) \partial_u H(u)du
		=-\res_{u=0}\log\Bigl(\frac{G(u)-\epsilon}{u}\Bigr) \partial_u H(u)du.
	\end{equation}
	The last equality holds since $\log(G(u)/u)$ is regular and does not contribute to the residue. This computation is applicable if $H(u)$ has no constant term, that is, if $H(0)=0$. Substituting $\epsilon=-W(0)$ we obtain the first equality~\eqref{eqB:L-B} in the case $H(0)=0$. In the case of an arbitrary value of $H(0)$ its contribution on the right hand side of~\eqref{eqB:L-B} should be added manually or one can rewrite the obtained equality as in the second line of~\eqref{eqB:L-B}; in this form it holds with no restriction on~$H(0)$.
\end{proof}

This proposition allows one to complete the computation of the constants $\fc^*_k$, $s^*_k$ since they are specializations at $u=u^*$ of certain functions in~$u$. Indeed, by definition, see~\eqref{eq:hkstarorigdef}--\eqref{eq:skorigdef}, we have
\begin{align}
\fc_k^*&=\sum_{m=0}^\infty \fc_{k+m}\frac{{u^*}^m}{m!}=\frac{\delta_{k,1}}{r}-\partial^k_uW(u)\bigm|_{u=u^*},\quad k\ge1,
\\e^{-\sum_{k-1}^\infty s^*_kz^k}&=1-\frac{1}{\frac1r-\fc_1^*}\sum_{i=1}^\infty \fc^*_{i+1}z^i
=\frac{\frac1r-\sum_{i=0}^\infty \fc^*_{i+1}z^i}{\frac1r-\fc_1^*}
\\\notag&=\frac{\sum_{i=0}^\infty z^i\partial_u^{i+1}W(u)}{\partial_uW(u)}\Bigm|_{u=u^*}
=\frac{\frac{\partial_u}{1-z\,\partial_u}W(u)}{\partial_uW(u)}\Bigm|_{u=u^*}.
\end{align}

\begin{corollary}\label{corB:res-formulas}
We have
\begin{align}\label{eqB:LB1}
u^*&=\res_{u=0}u\;\partial_u\!\log\bigl(W(u)\bigr)\,du
=-\res_{u=0}\log\bigl(W(u)/u\bigr)\;du,
\\\frac{1}{r}-\fc_1^*&=\res_{u=0}\partial_uW(u)\;\partial_u\!\log\bigl(W(u)\bigr)\,du,
\\\fc_k^*&=-\res_{u=0}\partial^k_uW(u)\;\partial_u\!\log\bigl(W(u)\bigr)\,du,\quad k\ge2,
\\\label{eqB:LB4}
\sum_{k=1}^\infty s^*_k z^k&=\res_{u=0}
\log\Bigl(\frac{\partial_uW(u)}{\frac{\partial_u}{1-z\,\partial_u}W(u)}\Bigr)\;
  \partial_u\!\log\bigl(W(u)\bigr)\,du.
\end{align}
\end{corollary}

The key advantage of these relations is that they can be computed for any choice of the local coordinate on the spectral curve; the residues are independent of a choice of the local coordinate. Eqs.~\eqref{eqB:uw}--\eqref{eqB:skw} of Proposition~\ref{prop:elsvformulazr} reproduce these relations in the local coordinate $w=t+1$, for clarity.

\printbibliography

@article {ABDKS-rationalspectralKP,
	AUTHOR = {Alexandrov, A. and Bychkov, B. and Dunin-Barkowski, P. and
	Kazarian, M. and Shadrin, S.},
	TITLE = {Any {T}opological {R}ecursion on a {R}ational {S}pectral
	{C}urve is {KP} {I}ntegrable},
	JOURNAL = {Comm. Math. Phys.},
	FJOURNAL = {Communications in Mathematical Physics},
	VOLUME = {407},
	YEAR = {2026},
	NUMBER = {4},
	PAGES = {Paper No. 69},
	ISSN = {0010-3616,1432-0916},
	MRCLASS = {99-06},
	MRNUMBER = {5041856},
	DOI = {10.1007/s00220-026-05566-9},
	URL = {https://doi.org/10.1007/s00220-026-05566-9},
}

@article {ABDKS-degenerate-irregular,
	AUTHOR = {Alexandrov, A. and Bychkov, B. and Dunin-Barkowski, P. and
	Kazarian, M. and Shadrin, S.},
	TITLE = {Degenerate and irregular topological recursion},
	JOURNAL = {Comm. Math. Phys.},
	FJOURNAL = {Communications in Mathematical Physics},
	VOLUME = {406},
	YEAR = {2025},
	NUMBER = {5},
	PAGES = {Paper No. 94, 31},
	ISSN = {0010-3616,1432-0916},
	MRCLASS = {14H81 (05E14 14H70 81T45)},
	MRNUMBER = {4887613},
	MRREVIEWER = {Reinier\ Kramer},
	DOI = {10.1007/s00220-025-05274-w},
	URL = {https://doi.org/10.1007/s00220-025-05274-w},
}

@article {ABDKS4,
	AUTHOR = {Alexandrov, A. and Bychkov, B. and Dunin-Barkowski, P. and
	Kazarian, M. and Shadrin, S.},
	TITLE = {Log topological recursion through the prism of {$x$}-{$y$}
	swap},
	JOURNAL = {Int. Math. Res. Not. IMRN},
	FJOURNAL = {International Mathematics Research Notices. IMRN},
	YEAR = {2024},
	NUMBER = {21},
	PAGES = {13461--13487},
	ISSN = {1073-7928,1687-0247},
	MRCLASS = {14H81},
	MRNUMBER = {4819863},
	MRREVIEWER = {Vida\ Milani},
	DOI = {10.1093/imrn/rnae213},
	URL = {https://doi.org/10.1093/imrn/rnae213},
}

@article {ABDKS-sympl,
	AUTHOR = {Alexandrov, Alexander and Bychkov, Boris and Dunin-Barkowski,
	Petr and Kazarian, Maxim and Shadrin, Sergey},
	TITLE = {Symplectic duality via log topological recursion},
	JOURNAL = {Commun. Number Theory Phys.},
	FJOURNAL = {Communications in Number Theory and Physics},
	VOLUME = {18},
	YEAR = {2024},
	NUMBER = {4},
	PAGES = {795--841},
	ISSN = {1931-4523,1931-4531},
	MRCLASS = {14H81 (05E14 14H30 14N10 37K10)},
	MRNUMBER = {4836057},
	MRREVIEWER = {Ali\ Shojaei-Fard},
	DOI = {10.4310/cntp.241203001416},
	URL = {https://doi.org/10.4310/cntp.241203001416},
}

@article {ACEH,
	AUTHOR = {Alexandrov, A. and Chapuy, G. and Eynard, B. and Harnad, J.},
	TITLE = {Weighted {H}urwitz numbers and topological recursion},
	JOURNAL = {Comm. Math. Phys.},
	FJOURNAL = {Communications in Mathematical Physics},
	VOLUME = {375},
	YEAR = {2020},
	NUMBER = {1},
	PAGES = {237--305},
	ISSN = {0010-3616},
	MRCLASS = {14H81 (14H30 33C20 33E30 37K10)},
	MRNUMBER = {4082183},
	MRREVIEWER = {Reinier Kramer},
	DOI = {10.1007/s00220-020-03717-0},
	URL = {https://doi.org/10.1007/s00220-020-03717-0},
}

@article{als,
	AUTHOR = {Alexandrov, A. and Lewanski, D. and Shadrin, S.},
	TITLE = {Ramifications of {H}urwitz theory, {KP} integrability and
	quantum curves},
	JOURNAL = {J. High Energy Phys.},
	FJOURNAL = {Journal of High Energy Physics},
	YEAR = {2016},
	NUMBER = {5},
	PAGES = {124, front matter+30},
	ISSN = {1126-6708},
	MRCLASS = {81P10},
	MRNUMBER = {3521843},
	DOI = {10.1007/JHEP05(2016)124},
	URL = {https://doi.org/10.1007/JHEP05(2016)124},
}

@article {banerjee2026gwdtinvariants5dbps,
    AUTHOR = {Banerjee, Sibasish and Hock, Alexander and Marchal, Olivier},
     TITLE = {G{W}/{DT} invariants and 5{D} {BPS} indices for strips from
              topological recursion},
   JOURNAL = {Lett. Math. Phys.},
  FJOURNAL = {Letters in Mathematical Physics},
    VOLUME = {116},
      YEAR = {2026},
    NUMBER = {1},
     PAGES = {Paper No. 14, 25},
      ISSN = {0377-9017,1573-0530},
   MRCLASS = {14N35 (05A15 14H81 30F30 81T30)},
  MRNUMBER = {5022421},
       DOI = {10.1007/s11005-026-02046-y},
       URL = {https://doi-org.libproxy.ibs.re.kr/10.1007/s11005-026-02046-y},
}

@article{banerjee2025quantumcurvestripgeometries,
   AUTHOR = {Banerjee, Sibasish and Hock, Alexander},
TITLE = {Quantum curve for strip geometries, topological recursion and
	open {GW}/{DT} invariants},
JOURNAL = {Lett. Math. Phys.},
FJOURNAL = {Letters in Mathematical Physics},
VOLUME = {116},
YEAR = {2026},
NUMBER = {1},
PAGES = {Paper No. 22},
ISSN = {0377-9017,1573-0530},
MRCLASS = {14H81 (05A15 30F30 81T30)},
MRNUMBER = {5037182},
DOI = {10.1007/s11005-026-02059-7},
URL = {https://doi.org/10.1007/s11005-026-02059-7},
}

@article {BDKLM,
    AUTHOR = {Borot, Ga\"{e}tan and Do, Norman and Karev, Maksim and
              Lewa\'{n}ski, Danilo and Moskovsky, Ellena},
     TITLE = {Double {H}urwitz numbers: polynomiality, topological recursion
              and intersection theory},
   JOURNAL = {Math. Ann.},
  FJOURNAL = {Mathematische Annalen},
    VOLUME = {387},
      YEAR = {2023},
    NUMBER = {1-2},
     PAGES = {179--243},
      ISSN = {0025-5831,1432-1807},
   MRCLASS = {14H30 (05A15 14N10 51P05)},
  MRNUMBER = {4631045},
       DOI = {10.1007/s00208-022-02457-x},
       URL = {https://doi.org/10.1007/s00208-022-02457-x},
}

@article {BorotKarevLewanski,
	AUTHOR = {Borot, Ga\"etan and Karev, Maksim and Lewa\'nski, Danilo},
	TITLE = {On {ELSV}-type formulae and relations between {$\Omega
	$}-integrals via deformations of spectral curves},
	JOURNAL = {J. Geom. Phys.},
	FJOURNAL = {Journal of Geometry and Physics},
	VOLUME = {207},
	YEAR = {2025},
	PAGES = {Paper No. 105343, 30},
	ISSN = {0393-0440,1879-1662},
	MRCLASS = {14H81 (05A15 14H10 14H60 14H70 14N10 81R10)},
	MRNUMBER = {4816550},
	MRREVIEWER = {Olivier\ Marchal},
	DOI = {10.1016/j.geomphys.2024.105343},
	URL = {https://doi.org/10.1016/j.geomphys.2024.105343},
}

@article {BS-blobbed,
	AUTHOR = {Borot, Ga\"{e}tan and Shadrin, Sergey},
	TITLE = {Blobbed topological recursion: properties and applications},
	JOURNAL = {Math. Proc. Cambridge Philos. Soc.},
	FJOURNAL = {Mathematical Proceedings of the Cambridge Philosophical
	Society},
	VOLUME = {162},
	YEAR = {2017},
	NUMBER = {1},
	PAGES = {39--87},
	ISSN = {0305-0041},
	MRCLASS = {14H81 (32G99 81T30)},
	MRNUMBER = {3581899},
	MRREVIEWER = {Xiaobin Li},
	DOI = {10.1017/S0305004116000323},
	URL = {https://doi.org/10.1017/S0305004116000323},
}

@Article{BDKS1,
	Author = {Bychkov, Boris and Dunin-Barkowski, Petr and Kazarian, Maxim and Shadrin, Sergey},
	Title = {Explicit closed algebraic formulas for {Orlov}-{Scherbin} {{\(n\)}}-point functions},
	FJournal = {Journal de l'{\'E}cole Polytechnique -- Math{\'e}matiques},
	Journal = {J. {\'E}c. Polytech., Math.},
	ISSN = {2429-7100},
	Volume = {9},
	Pages = {1121--1158},
	Year = {2022},
	DOI = {10.5802/jep.202},
	URL = {10.5802/jep.202},
	zbMATH = {7559604},
	Zbl = {1504.37080}
}

@article{bychkov2021topological,
    AUTHOR = {Bychkov, Boris and Dunin-Barkowski, Petr and Kazarian, Maxim
and Shadrin, Sergey},
TITLE = {Topological recursion for {K}adomtsev-{P}etviashvili tau
functions of hypergeometric type},
JOURNAL = {J. Lond. Math. Soc. (2)},
FJOURNAL = {Journal of the London Mathematical Society. Second Series},
VOLUME = {109},
YEAR = {2024},
NUMBER = {6},
PAGES = {Paper No. e12946, 57},
ISSN = {0024-6107,1469-7750},
MRCLASS = {14H81 (14H10 37K10 53D45 81T45)},
MRNUMBER = {4760443},
MRREVIEWER = {Nathan\ Grieve},
DOI = {10.1112/jlms.12946},
URL = {https://doi.org/10.1112/jlms.12946},
}

@article {guay2017generating,
    AUTHOR = {Guay-Paquet, Mathieu and Harnad, J.},
     TITLE = {Generating functions for weighted {H}urwitz numbers},
   JOURNAL = {J. Math. Phys.},
  FJOURNAL = {Journal of Mathematical Physics},
    VOLUME = {58},
      YEAR = {2017},
    NUMBER = {8},
     PAGES = {083503, 28},
      ISSN = {0022-2488,1089-7658},
   MRCLASS = {14H81 (14H30 35Q41 81Q05)},
  MRNUMBER = {3683833},
MRREVIEWER = {Noriko\ Yui},
       DOI = {10.1063/1.4996574},
       URL = {https://doi-org.libproxy.ibs.re.kr/10.1063/1.4996574},
}

@incollection {Harnad,
	AUTHOR = {Harnad, J.},
	TITLE = {Weighted {H}urwitz numbers and hypergeometric
	{$\tau$}-functions: an overview},
	BOOKTITLE = {String-{M}ath 2014},
	SERIES = {Proc. Sympos. Pure Math.},
	VOLUME = {93},
	PAGES = {289--333},
	PUBLISHER = {Amer. Math. Soc., Providence, RI},
	YEAR = {2016},
	ISBN = {978-1-4704-1992-9},
	MRCLASS = {14H30 (05A15 14H81 14N10 20C30)},
	MRNUMBER = {3525997},
	MRREVIEWER = {Francesca\ Vetro},
	DOI = {10.1090/pspum/093/01610},
	URL = {https://doi.org/10.1090/pspum/093/01610},
}

@article {DKPS-qr-ELSV,
	AUTHOR = {Dunin-Barkowski, Petr and Kramer, Reinier and Popolitov,
	Alexandr and Shadrin, Sergey},
	TITLE = {Loop equations and a proof of {Z}vonkine's {$qr$}-{ELSV}
	formula},
	JOURNAL = {Ann. Sci. \'{E}c. Norm. Sup\'{e}r. (4)},
	FJOURNAL = {Annales Scientifiques de l'\'{E}cole Normale Sup\'{e}rieure. Quatri\`eme
	S\'{e}rie},
	VOLUME = {56},
	YEAR = {2023},
	NUMBER = {4},
	PAGES = {1199--1229},
	ISSN = {0012-9593},
	MRCLASS = {14H30 (05E14 14H10 14N10 20C30)},
	MRNUMBER = {4650153},
	DOI = {10.24033/asens.2553},
	URL = {https://doi.org/10.24033/asens.2553},
}

@article {DKOSS-ELSV,
	AUTHOR = {Dunin-Barkowski, P. and Kazarian, M. and Orantin, N. and
	Shadrin, S. and Spitz, L.},
	TITLE = {Polynomiality of {H}urwitz numbers, {B}ouchard-{M}ari\~{n}o
	conjecture, and a new proof of the {ELSV} formula},
	JOURNAL = {Adv. Math.},
	FJOURNAL = {Advances in Mathematics},
	VOLUME = {279},
	YEAR = {2015},
	PAGES = {67--103},
	ISSN = {0001-8708},
	MRCLASS = {14H30 (14H15 57M60)},
	MRNUMBER = {3345179},
	MRREVIEWER = {Milagros Izquierdo},
	DOI = {10.1016/j.aim.2015.03.016},
	URL = {https://doi.org/10.1016/j.aim.2015.03.016},
}

@article {DOSS,
	AUTHOR = {Dunin-Barkowski, P. and Orantin, N. and Shadrin, S. and Spitz,
	L.},
	TITLE = {Identification of the {G}ivental formula with the spectral
	curve topological recursion procedure},
	JOURNAL = {Comm. Math. Phys.},
	FJOURNAL = {Communications in Mathematical Physics},
	VOLUME = {328},
	YEAR = {2014},
	NUMBER = {2},
	PAGES = {669--700},
	ISSN = {0010-3616},
	MRCLASS = {81T45 (14N35 53D45)},
	MRNUMBER = {3199996},
	MRREVIEWER = {Wan Keng Cheong},
	DOI = {10.1007/s00220-014-1887-2},
	URL = {https://doi.org/10.1007/s00220-014-1887-2},
}

@article {Eynard-intersections,
	AUTHOR = {Eynard, B.},
	TITLE = {Invariants of spectral curves and intersection theory of
	moduli spaces of complex curves},
	JOURNAL = {Commun. Number Theory Phys.},
	FJOURNAL = {Communications in Number Theory and Physics},
	VOLUME = {8},
	YEAR = {2014},
	NUMBER = {3},
	PAGES = {541--588},
	ISSN = {1931-4523},
	MRCLASS = {14H10 (11G05 14C17 32G15)},
	MRNUMBER = {3282995},
	MRREVIEWER = {Letterio Gatto},
	DOI = {10.4310/CNTP.2014.v8.n3.a4},
	URL = {https://doi.org/10.4310/CNTP.2014.v8.n3.a4},
}

@article {EO-1st,
	AUTHOR = {Eynard, B. and Orantin, N.},
	TITLE = {Invariants of algebraic curves and topological expansion},
	JOURNAL = {Commun. Number Theory Phys.},
	FJOURNAL = {Communications in Number Theory and Physics},
	VOLUME = {1},
	YEAR = {2007},
	NUMBER = {2},
	PAGES = {347--452},
	ISSN = {1931-4523},
	MRCLASS = {14H15 (14N35 32A27 37K10 37K20 81T45)},
	MRNUMBER = {2346575},
	MRREVIEWER = {Vincent Bouchard},
	DOI = {10.4310/CNTP.2007.v1.n2.a4},
	URL = {https://doi.org/10.4310/CNTP.2007.v1.n2.a4},
}

@article {hock2023xy,
	AUTHOR = {Hock, Alexander},
	TITLE = {{$x - y$} duality in topological recursion for exponential
	variables via quantum dilogarithm},
	JOURNAL = {SciPost Phys.},
	FJOURNAL = {SciPost Physics},
	VOLUME = {17},
	YEAR = {2024},
	NUMBER = {2},
	PAGES = {Paper No. 065, 36},
	ISSN = {2542-4653},
	MRCLASS = {14H81 (05E14 32G34 33B15 81T30)},
	MRNUMBER = {4794300},
	MRREVIEWER = {Ali\ Shojaei-Fard},
	URL = {doi:10.21468/SciPostPhys.17.2.065},
	DOI = {10.21468/SciPostPhys.17.2.065},
}

@article {Hock-SymplecticNonInvariance,
	AUTHOR = {Hock, Alexander},
	TITLE = {Symplectic (non-)invariance of the free energy in topological
	recursion},
	JOURNAL = {Comm. Math. Phys.},
	FJOURNAL = {Communications in Mathematical Physics},
	VOLUME = {406},
	YEAR = {2025},
	NUMBER = {8},
	PAGES = {Paper No. 192, 28},
	ISSN = {0010-3616,1432-0916},
	MRCLASS = {14H81 (30F99)},
	MRNUMBER = {4927821},
	DOI = {10.1007/s00220-025-05373-8},
	URL = {https://doi.org/10.1007/s00220-025-05373-8},
}

@misc{HockSha,
	title={Quantum Curves in the Context of Symplectic Duality}, 
	author={Alexander Hock and Sergey Shadrin},
	year={2025},
	eprint={2504.14924},
	archivePrefix={arXiv},
	primaryClass={math-ph},
	url={https://arxiv.org/abs/2504.14924}, 
}

@article {LPSZ,
	AUTHOR = {Lewanski, Danilo and Popolitov, Alexandr and Shadrin, Sergey
	and Zvonkine, Dimitri},
	TITLE = {Chiodo formulas for the {$r$}-th roots and topological
	recursion},
	JOURNAL = {Lett. Math. Phys.},
	FJOURNAL = {Letters in Mathematical Physics},
	VOLUME = {107},
	YEAR = {2017},
	NUMBER = {5},
	PAGES = {901--919},
	ISSN = {0377-9017},
	MRCLASS = {14H10 (14N10 14N35)},
	MRNUMBER = {3633029},
	MRREVIEWER = {Fabio Perroni},
	DOI = {10.1007/s11005-016-0928-5},
	URL = {https://doi.org/10.1007/s11005-016-0928-5},
}

@article {SSZ,
    AUTHOR = {Shadrin, S. and Spitz, L. and Zvonkine, D.},
     TITLE = {Equivalence of {ELSV} and {B}ouchard-{M}ari\~{n}o conjectures
              for {$r$}-spin {H}urwitz numbers},
   JOURNAL = {Math. Ann.},
  FJOURNAL = {Mathematische Annalen},
    VOLUME = {361},
      YEAR = {2015},
    NUMBER = {3-4},
     PAGES = {611--645},
      ISSN = {0025-5831,1432-1807},
   MRCLASS = {14H10 (53D45)},
  MRNUMBER = {3319543},
MRREVIEWER = {Felix\ Janda},
       DOI = {10.1007/s00208-014-1082-y},
       URL = {https://doi.org/10.1007/s00208-014-1082-y},
}

@article {rELSV-2,
	AUTHOR = {Borot, Ga\"{e}tan and Kramer, Reinier and Lewanski, Danilo and
	Popolitov, Alexandr and Shadrin, Sergey},
	TITLE = {Special cases of the orbifold version of {Z}vonkine's
	{$r$}-{ELSV} formula},
	JOURNAL = {Michigan Math. J.},
	FJOURNAL = {Michigan Mathematical Journal},
	VOLUME = {70},
	YEAR = {2021},
	NUMBER = {2},
	PAGES = {369--402},
	ISSN = {0026-2285},
	MRCLASS = {14N10 (05E10 14N35 30F30)},
	MRNUMBER = {4278700},
	MRREVIEWER = {Hsian-Hua Tseng},
	DOI = {10.1307/mmj/1592877614},
	URL = {https://doi.org/10.1307/mmj/1592877614},
}

@article {rELSV-1,
	AUTHOR = {Kramer, R. and Lewanski, D. and Popolitov, A. and Shadrin, S.},
	TITLE = {Towards an orbifold generalization of {Z}vonkine's
	{$r$}-{ELSV} formula},
	JOURNAL = {Trans. Amer. Math. Soc.},
	FJOURNAL = {Transactions of the American Mathematical Society},
	VOLUME = {372},
	YEAR = {2019},
	NUMBER = {6},
	PAGES = {4447--4469},
	ISSN = {0002-9947},
	MRCLASS = {14H30 (14H10 14H81 14N10 53D45)},
	MRNUMBER = {4009392},
	MRREVIEWER = {Navid Nabijou},
	DOI = {10.1090/tran/7793},
	URL = {https://doi.org/10.1090/tran/7793},
}

@article {JPT-1,
	AUTHOR = {Dunin-Barkowski, P. and Lewanski, D. and Popolitov, A. and
	Shadrin, S.},
	TITLE = {Polynomiality of orbifold {H}urwitz numbers, spectral curve,
	and a new proof of the {J}ohnson-{P}andharipande-{T}seng
	formula},
	JOURNAL = {J. Lond. Math. Soc. (2)},
	FJOURNAL = {Journal of the London Mathematical Society. Second Series},
	VOLUME = {92},
	YEAR = {2015},
	NUMBER = {3},
	PAGES = {547--565},
	ISSN = {0024-6107},
	MRCLASS = {14H10 (14H30 32G99 37K10)},
	MRNUMBER = {3431649},
	MRREVIEWER = {Andrei B. Bogatyr\"{e}v},
	DOI = {10.1112/jlms/jdv047},
	URL = {https://doi.org/10.1112/jlms/jdv047},
}

@article {JPT,
	AUTHOR = {Johnson, P. and Pandharipande, R. and Tseng, H.-H.},
	TITLE = {Abelian {H}urwitz-{H}odge integrals},
	JOURNAL = {Michigan Math. J.},
	FJOURNAL = {Michigan Mathematical Journal},
	VOLUME = {60},
	YEAR = {2011},
	NUMBER = {1},
	PAGES = {171--198},
	ISSN = {0026-2285},
	MRCLASS = {14D23 (14H10 14H30)},
	MRNUMBER = {2785870},
	MRREVIEWER = {Arvid Siqveland},
	DOI = {10.1307/mmj/1301586310},
	URL = {https://doi.org/10.1307/mmj/1301586310},
}

@article {Chiodo,
	AUTHOR = {Chiodo, Alessandro},
	TITLE = {Towards an enumerative geometry of the moduli space of twisted
	curves and {$r$}th roots},
	JOURNAL = {Compos. Math.},
	FJOURNAL = {Compositio Mathematica},
	VOLUME = {144},
	YEAR = {2008},
	NUMBER = {6},
	PAGES = {1461--1496},
	ISSN = {0010-437X},
	MRCLASS = {14H10 (14C17 14H60 14N10 14N35)},
	MRNUMBER = {2474317},
	MRREVIEWER = {Hsian-Hua Tseng},
	DOI = {10.1112/S0010437X08003709},
	URL = {https://doi.org/10.1112/S0010437X08003709},
}

@article {JPPZ,
	AUTHOR = {Janda, F. and Pandharipande, R. and Pixton, A. and Zvonkine,
	D.},
	TITLE = {Double ramification cycles on the moduli spaces of curves},
	JOURNAL = {Publ. Math. Inst. Hautes \'{E}tudes Sci.},
	FJOURNAL = {Publications Math\'{e}matiques. Institut de Hautes \'{E}tudes
	Scientifiques},
	VOLUME = {125},
	YEAR = {2017},
	PAGES = {221--266},
	ISSN = {0073-8301},
	MRCLASS = {14H10 (14N35)},
	MRNUMBER = {3668650},
	MRREVIEWER = {Emily Clader},
	DOI = {10.1007/s10240-017-0088-x},
	URL = {https://doi.org/10.1007/s10240-017-0088-x},
}

@article {GLN,
	AUTHOR = {Giacchetto, Alessandro and Lewa\'{n}ski, Danilo and Norbury, Paul},
	TITLE = {An intersection-theoretic proof of the {H}arer-{Z}agier
	formula},
	JOURNAL = {Algebr. Geom.},
	FJOURNAL = {Algebraic Geometry},
	VOLUME = {10},
	YEAR = {2023},
	NUMBER = {2},
	PAGES = {130--147},
	ISSN = {2313-1691},
	MRCLASS = {14N10 (05A15 14H10 14H60)},
	MRNUMBER = {4562996},
	MRREVIEWER = {Trygve Johnsen},
	URL = {doi:10.14231/AG-2023-004},
	DOI = {10.14231/AG-2023-004},
}

@article {ELSV,
	AUTHOR = {Ekedahl, Torsten and Lando, Sergei and Shapiro, Michael and
	Vainshtein, Alek},
	TITLE = {Hurwitz numbers and intersections on moduli spaces of curves},
	JOURNAL = {Invent. Math.},
	FJOURNAL = {Inventiones Mathematicae},
	VOLUME = {146},
	YEAR = {2001},
	NUMBER = {2},
	PAGES = {297--327},
	ISSN = {0020-9910},
	MRCLASS = {14H30 (14H10 14N35)},
	MRNUMBER = {1864018},
	MRREVIEWER = {Ravi D. Vakil},
	DOI = {10.1007/s002220100164},
	URL = {https://doi.org/10.1007/s002220100164},
}

@article {Givental-semisimple,
	AUTHOR = {Givental, Alexander B.},
	TITLE = {Semisimple {F}robenius structures at higher genus},
	JOURNAL = {Internat. Math. Res. Notices},
	FJOURNAL = {International Mathematics Research Notices},
	YEAR = {2001},
	NUMBER = {23},
	PAGES = {1265--1286},
	ISSN = {1073-7928,1687-0247},
	MRCLASS = {53D45 (14N35)},
	MRNUMBER = {1866444},
	MRREVIEWER = {Gilberto\ Bini},
	DOI = {10.1155/S1073792801000605},
	URL = {https://doi.org/10.1155/S1073792801000605},
}

@article {Janda-P1,
	AUTHOR = {Janda, Felix},
	TITLE = {Relations on {$\overline M_{g,n}$} via equivariant
	{G}romov-{W}itten theory of {$\mathbb{P}^1$}},
	JOURNAL = {Algebr. Geom.},
	FJOURNAL = {Algebraic Geometry},
	VOLUME = {4},
	YEAR = {2017},
	NUMBER = {3},
	PAGES = {311--336},
	ISSN = {2313-1691,2214-2584},
	MRCLASS = {14H10 (14N35)},
	MRNUMBER = {3652083},
	MRREVIEWER = {Martijn\ Kool},
	DOI = {10.14231/AG-2017-018},
	URL = {https://doi.org/10.14231/AG-2017-018},
}

@misc{alexandrov2025newspinpolynomialrelations,
	title={A new spin on polynomial relations among kappa classes}, 
	author={Alexander Alexandrov and Boris Bychkov and Petr Dunin-Barkowski and Maxim Kazarian and Sergey Shadrin},
	year={2025},
	eprint={2508.17865},
	archivePrefix={arXiv},
	primaryClass={math.AG},
	url={https://arxiv.org/abs/2508.17865}, 
}

@misc{blot2025cohomologicalrepresentationsquantumtau,
	title={Cohomological representations of quantum tau functions}, 
	author={Xavier Blot and Danilo Lewański and Sergey Shadrin},
	year={2025},
	eprint={2411.03499},
	archivePrefix={arXiv},
	primaryClass={math.AG},
	url={https://arxiv.org/abs/2411.03499}, 
}

@misc{BLS-inprep,
	title={A new integrable system associated to a cohomological field theory with a unit (tentative title)}, 
	author={Xavier Blot and Danilo Lewański and Sergey Shadrin},
	year={2026},
	addendum={In preparation.},
}

@article {Bini,
	AUTHOR = {Bini, Gilberto},
	TITLE = {Generalized {H}odge classes on the moduli space of curves},
	JOURNAL = {Beitr\"age Algebra Geom.},
	FJOURNAL = {Beitr\"age zur Algebra und Geometrie. Contributions to Algebra
	and Geometry},
	VOLUME = {44},
	YEAR = {2003},
	NUMBER = {2},
	PAGES = {559--565},
	ISSN = {0138-4821},
	MRCLASS = {14H10 (14C40)},
	MRNUMBER = {2017857},
	MRREVIEWER = {John\ B.\ Little},
	url = {http://eudml.org/doc/227725},
}

@article {GiaKraLew,
	AUTHOR = {Giacchetto, Alessandro and Kramer, Reinier and Lewa\'{n}ski,
	Danilo},
	TITLE = {A new spin on {H}urwitz theory and {ELSV} via theta
	characteristics},
	JOURNAL = {Selecta Math. (N.S.)},
	FJOURNAL = {Selecta Mathematica. New Series},
	VOLUME = {31},
	YEAR = {2025},
	NUMBER = {5},
	PAGES = {Paper No. 90, 83},
	ISSN = {1022-1824,1420-9020},
	MRCLASS = {14N10 (05A15 14H10 14N35 81R12)},
	MRNUMBER = {4963700},
	DOI = {10.1007/s00029-025-01077-y},
	URL = {https://doi.org/10.1007/s00029-025-01077-y},
}

@incollection {OrlovSch-1,
	AUTHOR = {Orlov, A. Yu. and Scherbin, D. M.},
	TITLE = {Multivariate hypergeometric functions as {$\tau$}-functions of
	{T}oda lattice and {K}adomtsev-{P}etviashvili equation},
	NOTE = {Advances in nonlinear mathematics and science},
	JOURNAL = {Phys. D},
	FJOURNAL = {Physica D. Nonlinear Phenomena},
	VOLUME = {152/153},
	YEAR = {2001},
	PAGES = {51--65},
	ISSN = {0167-2789,1872-8022},
	MRCLASS = {37K20 (33C80 33D80 37K05 37K30 39A12)},
	MRNUMBER = {1837897},
	MRREVIEWER = {Piotr\ G.\ Grinevich},
	DOI = {10.1016/S0167-2789(01)00158-0},
	URL = {https://doi.org/10.1016/S0167-2789(01)00158-0},
}

@article {OrlovSch-2,
	AUTHOR = {Orlov, A. Yu. and Shcherbin, D. M.},
	TITLE = {Hypergeometric solutions of soliton equations},
	JOURNAL = {Teoret. Mat. Fiz.},
	FJOURNAL = {Teoreticheskaya i Matematicheskaya Fizika},
	VOLUME = {128},
	YEAR = {2001},
	NUMBER = {1},
	PAGES = {84--108},
	ISSN = {0564-6162,2305-3135},
	MRCLASS = {37K10 (33C70 35Q51 35Q53)},
	MRNUMBER = {1904047},
	MRREVIEWER = {Alexander\ V.\ Shapovalov},
	DOI = {10.1023/A:1010402200567},
	URL = {https://doi.org/10.1023/A:1010402200567},
}

@article {alexandrov2025blobbedtopologicalrecursionkp,
	AUTHOR = {Alexandrov, A. and Bychkov, B. and Dunin-Barkowski, P. and
	Kazarian, M. and Shadrin, S.},
	TITLE = {Blobbed topological recursion and {KP} integrability},
	JOURNAL = {Selecta Math. (N.S.)},
	FJOURNAL = {Selecta Mathematica. New Series},
	VOLUME = {32},
	YEAR = {2026},
	NUMBER = {2},
	PAGES = {Paper No. 25},
	ISSN = {1022-1824,1420-9020},
	MRCLASS = {37K10 (14H70 14H81 81T45)},
	MRNUMBER = {5039140},
	DOI = {10.1007/s00029-026-01135-z},
	URL = {https://doi.org/10.1007/s00029-026-01135-z},
}

@article {Arnold,
	AUTHOR = {Arnol'd, V. I.},
	TITLE = {Topological classification of complex trigonometric
	polynomials and the combinatorics of graphs with an identical
	number of vertices and edges},
	JOURNAL = {Funktsional. Anal. i Prilozhen.},
	FJOURNAL = {Funktsional\cprime ny\u{\i} Analiz i ego Prilozheniya},
	VOLUME = {30},
	YEAR = {1996},
	NUMBER = {1},
	PAGES = {1--17, 96},
	ISSN = {0374-1990,2305-2899},
	MRCLASS = {32S50 (05C30 14E20 20F36)},
	MRNUMBER = {1387484},
	MRREVIEWER = {Aleksandr\ G.\ Aleksandrov},
	DOI = {10.4213/faa501},
	URL = {https://doi.org/10.4213/faa501},
}

@article {KMMM,
	AUTHOR = {Kharchev, S. and Marshakov, A. and Mironov, A. and Morozov,
	A.},
	TITLE = {Generalized {K}azakov-{M}igdal-{K}ontsevich model: group
	theory aspects},
	JOURNAL = {Internat. J. Modern Phys. A},
	FJOURNAL = {International Journal of Modern Physics A. Particles and
	Fields. Gravitation. Cosmology. Astrophysics. Accelerator
	Physics},
	VOLUME = {10},
	YEAR = {1995},
	NUMBER = {14},
	PAGES = {2015--2051},
	ISSN = {0217-751X,1793-656X},
	MRCLASS = {81T30 (58F07 81T40)},
	MRNUMBER = {1332645},
	MRREVIEWER = {Henrik\ Aratyn},
	DOI = {10.1142/S0217751X9500098X},
	URL = {https://doi.org/10.1142/S0217751X9500098X},
}

@article {Lagrange,
	AUTHOR = {Gessel, Ira M.},
	TITLE = {Lagrange inversion},
	JOURNAL = {J. Combin. Theory Ser. A},
	FJOURNAL = {Journal of Combinatorial Theory. Series A},
	VOLUME = {144},
	YEAR = {2016},
	PAGES = {212--249},
	ISSN = {0097-3165,1096-0899},
	MRCLASS = {05A15 (05A19)},
	MRNUMBER = {3534068},
	MRREVIEWER = {L\'aszl\'o\ T\'oth},
	DOI = {10.1016/j.jcta.2016.06.018},
	URL = {https://doi.org/10.1016/j.jcta.2016.06.018},
}

@unpublished{Zvonkine,
	title={A preliminary text on the r-{ELSV} formula},
	author={Zvonkine, Dimitri},
	year={2006},
	note={Preprint},
}

\end{document}